\documentclass[12pt,a4paper]{amsart}
\usepackage{fancyhdr}
\usepackage{appendix}
\usepackage{amssymb,amscd,amsxtra,calc}
\usepackage{mathrsfs}
\usepackage{amsmath}
\usepackage{multirow}
\usepackage{verbatim}

\usepackage{amsthm}
\usepackage{enumitem}
\usepackage[T1]{fontenc}
\usepackage{geometry}
\usepackage{graphicx}
\usepackage{mathtools}
\usepackage{setspace}

\usepackage{float}
\usepackage{microtype}
\usepackage{verbatim}

\usepackage[all,cmtip]{xy}

\usepackage[colorlinks,linkcolor=red,anchorcolor=blue,citecolor=blue]{hyperref}

\input cyracc.def

\usepackage{tikz}
\usetikzlibrary{arrows}

\def\leftacts{\ \ensuremath{%
    \rotatebox[origin=c]{-90}{$\circlearrowright$}}\ }
\def\rightacts{\ \ensuremath{%
    \rotatebox[origin=c]{90}{$\circlearrowleft$}}\ }

\theoremstyle{plain}
    \newtheorem{thm}{Theorem}[section]
    
    \newtheorem{claim}[thm]{Claim}
     \newtheorem{conjecture}[thm]{Conjecture}
    \newtheorem{corollary}[thm]{Corollary}
    \newtheorem{example}[thm]{Example}
    \newtheorem{lemma}[thm]{Lemma}
    \newtheorem{proposition}[thm]{Proposition}
    \newtheorem{question}[thm]{Question}
    \newtheorem{theorem}[thm]{Theorem}

\theoremstyle{definition}

    \newtheorem{notation}[thm]{Notation}
    \newtheorem*{notation*}{Notation and Terminology}
    \newtheorem{remark}[thm]{Remark}

\theoremstyle{remark}

\makeatletter

\newcommand{\Rmnum}[1]{\expandafter\@slowromancap\romannumeral #1@}
\makeatother

\begin{document}

\title[Morphisms onto Fano manifolds]
{Boundedness of finite morphisms onto Fano manifolds with large Fano index}

\author{Feng Shao} 
\address
{
\textsc{Center for Complex Geometry, Institute for Basic Science (IBS), 55 Expo-ro,  Yuseong-gu, Daejeon, 34126, Republic of Korea.
}}
\email{shaofeng@amss.ac.cn, shaofeng@ibs.re.kr}
\author{Guolei Zhong}
\address
{
\textsc{Center for Complex Geometry, Institute for Basic Science (IBS), 55 Expo-ro,  Yuseong-gu, Daejeon, 34126, Republic of Korea.
}}
\email{zhongguolei@u.nus.edu, guolei@ibs.re.kr}

\begin{abstract}
Let $f:Y\to X$ be a finite  morphism between Fano manifolds $Y$ and $X$ such that the Fano index of $X$ is greater than 1. 
On the one hand, when both $X$ and $Y$ are fourfolds of Picard number 1, we show that the degree of $f$ is bounded in terms of $X$ and $Y$ unless $X\cong\mathbb{P}^4$; hence, such $X$ does not admit any non-isomorphic surjective endomorphism.
On the other hand, when $X=Y$ is either a fourfold or a del Pezzo manifold, we prove that, if $f$ is an int-amplified endomorphism, then $X$ is toric.
Moreover, we classify all the singular quadrics  admitting non-isomorphic  endomorphisms. 
\end{abstract}


\maketitle

\tableofcontents

\section{Introduction}
We work over the field $\mathbb{C}$ of complex numbers.
Given two smooth projective varieties $Y$ and $X$ of the same dimension, it is a natural topic to study  finite  morphisms  $f:Y\to X$ between them. 
If the canonical divisor $K_X$ is big, then the degree $\deg(f)$ is bounded in terms of the canonical divisors of $X$ and $Y$ (cf.~\cite{KO75}) and we say that such $f$ satisfies the \textit{boundedness property}. 
However, different from the varieties of general type, (weak) Calabi-Yau varieties and rationally connected varieties in general do not possess such a nice property any more. 
For example, if $Y$ is an abelian variety (resp. the projective space $\mathbb{P}^n$), then one can take $X=Y$ and take $f$ to be the multiplication map (resp. the coordinate map) such that the degree of $f$ can be  arbitrarily large. 

As suggested by Van de Ven, Amerik proposed the following conjecture in \cite{Ame97}. 
\begin{conjecture}[{\cite[Conjecture B]{Ame97}}]\label{main-conj-bounded}
Let $f:Y\to X$ be a finite  morphism between smooth projective varieties of dimension $n\ge 2$ and of the second Betti number $1$.
Suppose  $X\not\cong\mathbb{P}^n$. 
Then $\deg(f)$ is bounded in terms of the discrete invariants of  $Y$ and $X$. 
\end{conjecture}

To the best knowledge of the authors, Conjecture \ref{main-conj-bounded} is known if $\dim (X)\le 3$ (\cite{ARV99} and \cite{Ame97}), or $X$ is a smooth quadric of  $\dim(X)\ge 3$ (\cite{Ame07}). 
Since  the (strict) Calabi-Yau case has been settled in \cite[Proposition 2.1]{ARV99}, our main interest here is to consider the case when both $Y$ and $X$ are   \textit{Fano}, i.e., the anti-canonical divisors $-K_Y$ and $-K_X$ are ample. 
More precisely, we ask: 
\begin{question}\label{ques-main-bounded-fano}
Let $X$ and $Y$ be  Fano manifolds of dimension $n$ and of Picard number $1$. 
Suppose that $X\not\cong\mathbb{P}^n$. 
Does there exist a positive number $N$ which depends only on $X$ and $Y$ such that every finite  morphism $Y\to X$ has the degree no more than $N$?
\end{question}

 Question \ref{ques-main-bounded-fano} implies the following long-standing Conjecture \ref{main-conj-pn} of the 1980s, which has been proved for (almost) homogeneous spaces (\cite{PS89}, \cite{HN11}), for Fano threefolds of Picard number 1 (\cite{ARV99}, \cite{Ame97}, \cite{HM03}), for Fano manifolds of Picard number 1 containing a rational curve with trivial normal bundle (\cite{HM03}), and for smooth  hypersurfaces of the projective space (\cite[Proposition 8]{PS89}, \cite{Bea01}; cf.~Corollary \ref{cor-usual-hyper}).
\begin{conjecture}\label{main-conj-pn}
Let $X$ be a Fano manifold of Picard number 1. 
Suppose that $X$ admits a non-isomorphic surjective endomorphism.
Then $X$ is a projective space.
\end{conjecture}

It has been a long history for people to  classify projective varieties $X$ (not necessarily of Picard number 1) admitting a non-isomorphic endomorphism.
When $\dim (X)=2$, it has been well-understood by Nakayama (\cite{Nak20a} and \cite{Nak20b}).  
When  $\dim(X)\ge 3$, Meng and Zhang established the  \textit{equivariant minimal model program}, which provides us
the tools of endomorphism-descending along the  minimal model program  (\cite{MZ18} and \cite{Men20}).  Extending Conjecture \ref{main-conj-pn}, Fakhruddin asked the following question, which is the initial point of the second part of our paper.
Toric varieties having many interesting dynamical symmetries (cf.~\cite{Nak02}), Question \ref{main-ques-toric} can be seen as a converse direction. 

\begin{question}[{\cite[Question 4.4]{Fak03}, \cite[Question 1.2]{MZg23}}]\label{main-ques-toric}
Let $X$ be a rationally connected smooth projective variety. 
Suppose that $X$ admits a polarized or an int-amplified endomorphism $f$. 
Is $X$ a toric variety?
\end{question}

A surjective endomorphism  $f:X\to X$ of a  projective variety $X$ is  \textit{polarized} if $f^*H\sim qH$ for some ample Cartier divisor $H$ on $X$ and integer $q>1$, and \textit{int-amplified} if $f^*H-H$ is ample for some ample Cartier divisor $H$ on $X$; see  \cite{MZ18} and \cite{Men20}. 
Clearly, every non-isomorphic surjective endomorphism on a projective variety of Picard number 1 is polarized and every polarized endomorphism is int-amplified.
We note that morphisms onto toric varieties do not satisfy any boundedness property.

Question \ref{main-ques-toric} is known to have a positive answer in the following cases: (i) $\dim (X)=2$ (\cite[Theorem 3]{Nak02}); 
(ii)  $X$ is a Fano threefold (\cite[Theorem 1.4]{MZZ22}); or (iii) $X$ is a Fano fourfold admitting a conic bundle structure  (\cite[Theorem 1.4]{JZ23}).

In this paper, we  study Questions \ref{ques-main-bounded-fano}, \ref{main-ques-toric} and Conjecture \ref{main-conj-pn} for Fano manifolds with large Fano index. 
The \textit{Fano index} (or simply \textit{index})  of a Fano manifold $X$ is the maximal integer $i(X)$ such that $-K_X\sim i(X)H$ for some integral ample divisor $H$. 
By a well-known theorem of Kobayashi and Ochiai \cite{KO73}, the  index satisfies $i(X)\le \dim(X)+1$; 
if $i(X)=\dim (X)+1$, then $X$ is the projective space, and if $i(X)=\dim(X)$, then $X$ is a smooth quadric. 
Theorems \ref{rho=1-dim=4} $\sim$ \ref{thm-middle-index} below are our main results.

First, Theorem \ref{rho=1-dim=4}  gives a positive answer to Question \ref{ques-main-bounded-fano}  for Fano fourfolds $X$ with  index $i(X)>1$ and thus verifies the boundedness property for morphisms onto such $X$. 
We refer readers to Question \ref{ques-fourfold-bounded} and Remark \ref{rem-index=1-dim=4} for the case $i(X)=1$.

\begin{theorem}\label{rho=1-dim=4}
Let $X$ be a smooth Fano fourfold  of Picard number 1 and of  index $i(X)>1$, which is different from $\mathbb{P}^4$. 
Then for any smooth Fano fourfold $Y$ of Picard number 1, there is a positive integer $N$ such that any surjective morphism $Y\to X$  has the degree no more than $N$. 
In particular, $X$ does not admit any non-isomorphic surjective endomorphism.
\end{theorem} 

Recently, Kawakami and Totaro showed in  \cite{KT23} that if a normal projective variety \(X\) of Picard number 1 violates the boundedness property, then it satisfies Bott vanishing; see \cite[Proposition 2.7]{KT23}.
In this line, they gave a different approach towards the boundedness problem and 
proved a stronger version of our Theorem \ref{rho=1-dim=4} (cf.~\cite[Theorem B]{KT23}). 

Considering the dynamics on weighted projective hypersurfaces (cf.~Theorem \ref{thm-wps-ci} and Corollary \ref{coro-del-pezzo-bounded}), our next  result  confirms   Conjecture \ref{main-conj-pn} for del Pezzo manifolds (cf.~Remark \ref{rem-mukai-rho=1} for a partial answer on Mukai manifolds). 
Recall that a Fano manifold $X$ is said to be \textit{del Pezzo} (resp. \textit{Mukai}), if its index $i(X)=\dim(X)-1$ (resp. $i(X)=\dim(X)-2$).

\begin{theorem}\label{main-del-pezzo-non}
Let $X$ be a del Pezzo manifold of Picard number 1.
Then $X$ does not admit any non-isomorphic surjective endomorphism.
\end{theorem}

When studying Conjecture  \ref{main-conj-bounded}, we tried to extend \cite{Ame07} to singular quadrics.
Indeed, if a Fano manifold $X$ admits a finite cover over a quadric, then the boundedness property for morphisms onto quadrics would imply the boundedness property for morphisms onto $X$. 
For example, if $X$ is a smooth complete intersection of two quadrics, we use this   observation to confirm Conjecture  \ref{main-conj-bounded}  (cf.~Proposition \ref{pro-complete-intersection-bounded}).

Theorem \ref{thm-singular-quadric} below  gives an analogue to Conjecture \ref{main-conj-pn} which shows some evidence towards our expectation. 
Besides, Theorem \ref{thm-singular-quadric}  extends \cite[Proposition 8]{PS89} and has its own independent interests (cf.~Remark \ref{rem-quadric-totally}).

\begin{theorem}\label{thm-singular-quadric} 
Let $X$ be a quadric hypersurface in $\mathbb{P}^n$,  i.e., after a linear map, $X$ is defined by  $\sum_{i=0}^{k}x_i^2=0$ where $[x_0:\cdots:x_n]$ are the homogeneous coordinates of $\mathbb{P}^n$.
Then  $X$ does not admit any non-isomorphic surjective  endomorphism if and only if $k\ge 4$. 
\end{theorem}

In the following, we treat the case when the Picard number $\ge 2$  and  we answer Question \ref{main-ques-toric} affirmatively if the index is large (cf.~Theorem \ref{thm-higher-pic}).

\begin{theorem}\label{thm-middle-index}
Let $X$ be a  Fano manifold of Picard number $\ge 2$. 
Suppose that the index $i(X)\ge \lfloor\frac{\dim(X)+1}{2}\rfloor$ (which is the case if $X$ is  (i) a Fano fourfold of $i(X)>1$, or (ii) a del Pezzo manifold, or (iii) a Mukai manifold of dimension $\ge 4$). 
Then $X$ is toric if and only if $X$ admits an int-amplified endomorphism. 
\end{theorem}

Consequently, Theorem \ref{thm-middle-index}  together with Theorems \ref{rho=1-dim=4} and \ref{main-del-pezzo-non} gives a positive answer to Question \ref{main-ques-toric} for smooth Fano fourfolds of index $>1$, and for del Pezzo manifolds.
\begin{corollary}
Let $X$ be either a smooth Fano fourfold of index $>1$, or a del Pezzo manifold.
If $X$ admits an int-amplified endomorphism, then $X$ is toric. 
\end{corollary}

Finally, in terms of higher dimensional Fano manifolds of small index, we refer readers to Section \ref{sec-index<=2} for some  partial solutions. 
Let us end up the introductory section by asking the following question.
We note that a positive answer to Question \ref{ques-fourfold-bounded} will complement Theorem \ref{rho=1-dim=4} and answer Question \ref{ques-main-bounded-fano} 
completely in $\dim (X)=4$ (see Remark \ref{rem-index=1-dim=4}).

\begin{question}\label{ques-fourfold-bounded}
Let $X$ be a smooth Fano fourfold of Picard number 1 and of  index 1.
Suppose the variety of minimal rational tangents $\mathcal{C}_x$ along a general point $x\in X$ is 1-dimensional. 
For any smooth Fano fourfold $Y$ of Picard number 1, is there a positive number $N$ such that every surjective morphism  $Y\to X$  has the degree no more than $N$?
\end{question}

\subsubsection*{\textbf{\textup{Acknowledgments}}}
The authors would like to thank Professors  Cinzia Casagrande, Baohua Fu, Jaehyun Hong, Jun-Muk Hwang, Yongnam Lee,  Sheng Meng and  De-Qi Zhang for many valuable suggestions and  inspiring discussions.
The authors would also like to thank the referee for the careful reading and suggestions to improve the paper.
This work was supported by the Institute for Basic Science (IBS-R032-D1-2022-a00).

\section{Preliminary}
We use the following notation throughout this paper.
\begin{notation}\label{notation2.1}
Let $X$ be a projective variety.
\begin{itemize}
\item We denote by $\sim$ (resp.~$\equiv$) 
the \textit{linear equivalence} (resp.~\textit{numerical equivalence}) on  (resp. $\mathbb{R}$-) Cartier divisors.

\item Let $\textup{NS}(X) = \textup{Pic}(X)/\textup{Pic}^\circ(X)$  be the \textit{N\'eron-Severi group} of $X$.
Let
 $\textup{N}^1(X):=\textup{NS}(X)\otimes_\mathbb{Z}\mathbb{R}$ be the space of $\mathbb{R}$-Cartier divisors modulo  numerical equivalence and $\rho(X) :=\dim_{\mathbb{R}}\textup{N}^1(X)$ the
\textit{Picard number} of $X$.

\item A smooth projective variety $X$ is said to be a \textit{Fano} manifold if the anti-canonical divisor $-K_X$ is ample.
Let $X$ be a Fano manifold.
The \textit{Fano index} (or simply \textit{index}) of $X$ is defined as the maximal integer $i(X)$ such that $-K_X\sim i(X)H$ for some integral ample divisor $H$. 
If $\rho(X)=1$, we say that $H$ is the \textit{fundamental divisor} of $X$, i.e., $H$ generates the Picard group $\textup{Pic}(X)$.

\item A Fano manifold of dimension $n$ is said to be \textit{del Pezzo} (resp. \textit{Mukai}), if its index $i(X)=n-1$ (resp. $i(X)=n-2$).
Let $X$ be a del Pezzo (resp. Mukai) manifold of dimension $n\ge 3$. 
We define the \textit{degree} $d=H^n$ (resp. \textit{genus} $g=\frac{1}{2}H^n+1$) where $H$ is the fundamental divisor of $X$.
It is known that $1\le d\le 7$ (resp.~$2\le g\le 12$ and  $g\neq 11$), and if $d\ge 3$ (resp. $g\ge 4$), then $H$ is very ample (cf.~e.g.~\cite[Corollary 2.4.7 and Remark 3.3.2]{IP99}).

\item Denote by $G:=\textup{Gr}(k,n)$ the Grassmannian of $(k-1)$-dimensional linear subspaces in $\mathbb{P}^{n-1}$.
Note that there is a natural Pl\"ucker embedding
$G\hookrightarrow \mathbb{P}(\wedge^k\mathbb{C}^n)=\mathbb{P}^{\binom{n}{k}-1}$.
Any Grassmannian is a rational homogeneous (and hence   Fano) manifold.

\item A normal variety $X$ of dimension $n$ is a \textit{toric variety} if $X$ contains a {\it big torus} $T=(\mathbb{C}^*)^n$ as an (affine) open dense subset such that the natural multiplication action of $T$ on itself extends to an action on the whole variety. In this case, let $B:=X\backslash T$, which is a divisor; the pair $(X,B)$ is said to be a \textit{toric pair}.

\item Let $X$ be a quadric hypersurface in $\mathbb{P}^n$, and $[x_0:\cdots:x_n]$  the homogeneous coordinates of $\mathbb{P}^n$.
After a  linear map,  $X$  is defined by $\sum_{i=0}^kx_i^2=0$ (cf.~\cite{Rei72}). 
If $k=n$, such $X$ is smooth and we denote $X$ by $Q^{n-1}$. 
If $k<n$, then $X$ is a cone over a smooth quadric of dimension $(k-1)$ with the singular locus  a linear subspace of dimension $(n-1-k)$, and we say that $k$ is the  \textit{rank} of $X$. 

\item Let $a_0,\cdots,a_n$ be positive integers.
Consider the graded algebra $\mathbb{C}[x_0,\cdots,x_n]$ where the grading is defined by assigning the weights $a_i$ to the variables $x_i$.
Then $\mathbb{P}=\mathbb{P}(a_0,\cdots,a_n):=\textup{Proj}\,\mathbb{C}[x_0,\cdots,x_n]$ is said to be a \textit{weighted projective space with respect to the weights $(a_0,\cdots,a_n)$}. 

\item Let $\mathbb{P}=\mathbb{P}(a_0,\cdots,a_n)$ be a weighted projective space. 
We say that $\mathbb{P}$ is  \textit{well-formed} if the greatest common divisor of any $n$ of the weights $a_i$ is 1. 
Note that any weighted projective space is isomorphic to a well-formed one.
A subvariety $X\subseteq\mathbb{P}$ is said to be \textit{well-formed} if $\mathbb{P}$ is well-formed and $\textup{codim}_X(X\cap\textup{Sing}\,\mathbb{P})\ge 2$ where $\textup{Sing}\,\mathbb{P}$ is the singular locus of $\mathbb{P}$.
It is known that a smooth well-formed weighted hypersurface $X$ is disjoint with
$\textup{Sing}\,\mathbb{P}$. 
Moreover, if $a_0,...,a_n$ are mutually coprime, then $\mathbb{P}$ has only isolated singularities and $\mathcal{O}_{\mathbb{P}}(1)$ is locally free outside $\textup{Sing}\,\mathbb{P}$ (cf. \cite[Proposition 5.1, Lemma 5.3, Definition 5.4 and Proposition 5.5]{BR86}).

\item Let $X$ be a Fano manifold of $\rho(X)=1$. 
A smooth rational curve $\ell$ on $X$ is said to be a \textit{line} (resp.\,a conic), if $H\cdot\ell=1$ (resp. $H\cdot\ell=2$) holds where $H$ is the fundamental divisor of $X$.
If  $H$ is very ample, then such a line in $X$ is an actual line in the projective space.

\item Let $X$ be a uniruled projective manifold and RatCurves$^n (X)$ the normalized space of rational curves on $X$ (cf.~\cite[Chapter II]{Kol96}). 
An irreducible component $\mathcal{K}$ of RatCurves$^n (X)$ is called a \textit{family of rational curves}. 
Let $e:\mathcal{U}\to X$ be the normalization of the universal family over $\mathcal{K}$ and let $q: \mathcal{U}\to \mathcal{K}$ be the natural $\mathbb{P}^1$-fibration. 
We denote by $\mathcal{K}_x$ the normalization of the set $q(e^{-1}(x))$ for some point $x\in e(\mathcal{U})$. 
The family $\mathcal{K}$ is said to be \textit{dominating} if $e$ is dominant and is said to be \textit{locally unsplit} if $\mathcal{K}_x$ is proper for a general $x\in e(U)$.  

Every uniruled projective manifold $X$  carries a locally unsplit dominating family $\mathcal{K}$ of minimal rational curves. 
A general element $[\ell ]\in \mathcal{K}$ is a \textit{standard} rational curve, i.e., the pullback of $T_X$ to the normalization  $f:\mathbb{P}^1\to \ell$ splits as follows:
$$
f^*T_X\cong\mathcal{O}_{\mathbb{P}^1}(2)\oplus \mathcal{O}_{\mathbb{P}^1}(1)^{\oplus p}\oplus \mathcal{O}_{\mathbb{P}^1}^{\oplus(n-p-1)}
$$
where $n=\dim(X)$ and $p=\dim(\mathcal{K}_x)$ for a general point $x\in X$.

Associated with a locally unsplit dominating family $\mathcal{K}$ of minimal rational curves and a general point $x\in X$, there is an important geometric object defined as the closure of the image of the tangent map $\tau_x :\mathcal{K}_x\dashrightarrow \mathbb{P}(\Omega_{X,x})$ sending a rational curve that is smooth at $x$ to its tangent direction. 
This is called the \textit{variety of minimal rational tangents} (VMRT for short) at the point $x$ associated to the family $\mathcal{K}$ and we denote it by $\mathcal{C}_x$. The tangent map $\tau_x:\mathcal{K}_x\dashrightarrow \mathcal{C}_x$ is in fact the normalization morphism (cf.~\cite{Keb02} and \cite{HM04})  and hence the dimension of $\mathcal{C}_x$ is equal to the dimension of $\mathcal{K}_x$, which is exactly $p$.

\item Let $f:X\to X$ be a non-isomorphic surjective endomorphism. 
A closed subset $Y\subseteq X$ is said to be \textit{$f^{-1}$-invariant} or \textit{$f$-totally invariant} if $f^{-1}(Y)=Y$.

\item A surjective endomorphism $f:X\to X$ is \textit{$q$-polarized} if $f^*L\sim qL$ for some ample Cartier divisor $L$ and integer $q>1$ (cf.~\cite{MZ18}). 
A surjective endomorphism $f:X\to X$ is \textit{int-amplified} if $f^*L - L$ is ample for some ample Cartier divisor $L$; see \cite{Men20}.
Note that every polarized endomorphism is int-amplified. 
\end{itemize}
\end{notation}

In the following, we recall several results to be used in the subsequent sections.  
First, we slightly extend \cite[Lemma 1.1]{ARV99} to the lemma below, where $\Omega_X(H)$ here is allowed to be globally generated away from finitely many points.

\begin{lemma}[{cf.~\cite[Lemma 1.1 and Corollary 1.2]{ARV99}}]\label{lem-arv-ggg}
Let $f:Y\to X$ be a finite surjective morphism between smooth projective varieties of dimension $n$.
Let $H$ be a line bundle on $X$ such that $\Omega_X(H)$ is globally generated outside finitely many points.
Then for a general section  $s\in H^0(X,\Omega_X(H))$, the induced section $t_s\in H^0(Y,\Omega_Y(f^*H))$ has only isolated zeros.
In particular, $\deg(f)c_n(\Omega_X(H))\le c_n(\Omega_Y(f^*H))$.
\end{lemma}
\begin{proof}
Let 
$R_i:=\{\textup{closed points }y\in Y~|~\textup{rank }df_y\le i\}$. 
Since $f$ is a finite morphism, by \cite[Chapter III, Proposition 10.6]{Har77}, we have $\dim(R_i)\le i$. 
Under the natural map
$$f^*(\Omega_X(H))\xrightarrow{\tau} \Omega_Y(f^*H)\to \Omega_{Y/X}(f^*H)\to 0,$$
every  $s\in H^0(X,\Omega_X(H))$ gives rise to an induced section $t_s:=\tau(f^*s)\in H^0(Y,\Omega_Y(f^*H))$.
We note that the above sequence is also left exact, since $\tau$ is generically surjective and the kernel $\ker(\tau)$ is thus  a torsion sheaf and has to be zero.

Suppose for a general $s\in H^0(X,\Omega_X(H))$, the induced $t_s$ vanishes along some curve $C_s$.
There exists one $k$ such that for a general $s$, the curve $C_s\subseteq R_k$ but $C_s\not\subseteq R_{k-1}$.
With $R_k$ replaced by an irreducible component, we may  assume that $R_k$ is irreducible.
Let 
$$\mathcal{B}_k=\ker(f^*(\Omega_X(H))|_{R_k}\to \Omega_Y(f^*H)|_{R_k})$$
and let 
$\mathcal{A}_k:=f^*(\Omega_X(H))|_{R_k}/\mathcal{B}_k$ 
be its quotient.
Then each $s$ induces $s_{R_k}\in H^0(R_k, \mathcal{A}_k)$ via the quotient and $s_{R_k}$ vanishes along $C_s$.
By the choice of $R_k$, the sheaf $\mathcal{A}_k$ is locally free of rank $k$ outside $R_{k-1}$ and generated by global sections $s_{R_k}$ outside finitely many points.
By considering the degeneracy locus of the generic section  (cf.~e.g.~\cite[Lemma 5.2]{EH16}) and the unique extension of a reflexive sheaf, 
the zero locus of a generic section $s_{R_k}$ has codimension $k$ in  $R_k\backslash R_{k-1}$; in particular, we have $C_s\subseteq R_{k-1}$, a contradiction to our assumption.
The first part is proved.

For the second part, since $c_n(f^*(\Omega_X(H)))=\deg(f)c_n(\Omega_X(H))$, we only need to show $c_n(f^*(\Omega_X(H)))\le c_n(\Omega_Y(f^*H))$. 
However, this follows from the first part by noting that the induced section $t_s$ has no less (isolated) zeros than those of $f^*(s)$. 
\end{proof}

As an application of Lemma \ref{lem-arv-ggg},
we prove the following lemma, which will be crucially used in the next section.
\begin{lemma}\label{lem-arv-prop2.1}
Let $X$ and $Y$ be smooth projective varieties of dimension $n$ and of Picard number 1.
Let $H_X$ (resp. $H_Y$) be the ample generator of $\textup{N}^1(X)$ (resp. $\textup{N}^1(Y)$).
Suppose that $\Omega_X(uH_X)$ is globally generated outside finitely many points for some $u>0$, and $c_n(\Omega_X(uH_X))-u^nH_X^n>0$. 
Then there is a positive number $N$ determined by $X$ and $Y$ such that every finite surjective morphism $f:Y\to X$ has the degree no more than $N$.
\end{lemma}

\begin{proof}
Write $f^*H_X\equiv mH_Y$.
Since $\Omega_X(uH_X)$ is globally generated away from finitely many points, it follows from Lemma \ref{lem-arv-ggg} that
$$m^nH_Y^nc_n(\Omega_X(uH_X))=\deg(f)H_X^nc_n(\Omega_X(uH_X))\le H_X^nc_n(\Omega_Y(umH_Y)).$$
Therefore, we have
$$m^nH_Y^n\sum_{i=0}^nc_{n-i}(\Omega_X)(uH_X)^i\le H_X^n \sum_{i=0}^nc_{n-i}(\Omega_Y)(umH_Y)^i.$$
Cancelling the last summand, we have
$$m^nH_Y^n\sum_{i=0}^{n-1}c_{n-i}(\Omega_X)(uH_X)^i\le H_X^n \sum_{i=0}^{n-1}c_{n-i}(\Omega_Y)(umH_Y)^i.$$
If $m$ is unbounded, by letting $m\to\infty$, we have
$$c_n(\Omega_X(uH_X))-u^nH_X^n=\sum_{i=0}^{n-1}c_{n-i}(\Omega_X)(uH_X)^i\le0,$$
which leads to a contradiction to our assumption. 
This finishes the proof of our lemma.
\end{proof}

When proving Theorems \ref{rho=1-dim=4} $\sim$ \ref{thm-middle-index} and their corollaries, we  use the classification of del Pezzo manifolds more than once.
For the convenience of the reader, we recall it here.
\begin{theorem}[{cf.~e.g.~\cite[Theorem 3.3.1]{IP99}}]\label{class-del-pezzo}
Let $(X,H)$ be an $n$-dimensional del Pezzo manifold of degree $d=H^n$.
Assume that $n\ge 3$.
Then $X$ is one of the following.
\begin{itemize}
\item $d=1$ and $X$ is a smooth weighted hypersurface of degree 6 in $\mathbb{P}(3,2,1^{\oplus n})$.
\item $d=2$ and $X$ is a smooth weighted hypersurface of degree 4 in $\mathbb{P}(2,1^{\oplus (n+1)})$.
\item $d=3$ and $X$ is a smooth cubic hypersurface in $\mathbb{P}^{n+1}$.
\item $d=4$ and $X$ is a smooth complete intersection of two smooth quadrics in $\mathbb{P}^{n+2}$.
\item $d=5$, $n\le 6$ and $X\subseteq\mathbb{P}^{n+3}$ is a linear section of the Grassmannian $\textup{Gr}(2,5)\subseteq\mathbb{P}^9$ under the Pl\"ucker embedding.
\item $d=6$, $n\le 4$ and $X$ is $\mathbb{P}^2\times\mathbb{P}^2$, $\mathbb{P}(T_{\mathbb{P}^2})$ or $\mathbb{P}^1\times\mathbb{P}^1\times\mathbb{P}^1$.
\item $d=7$ and $X$ is a blow-up of a point on $\mathbb{P}^3$.
\end{itemize}
Moreover, the Picard number $\rho(X)=1$ if and only if $d\le 5$, and the fundamental divisor $H$ is very ample if and only if $d\ge 3$.
\end{theorem}

Now we study lines on del Pezzo manifolds  and Mukai fourfolds. 
It is well-known that any del Pezzo manifold of Picard number 1 and of dimension $\ge 3$  is covered by lines.

\begin{lemma}\label{lem-lines-del}
Let $X$ be a del Pezzo manifold of Picard number 1. 
Suppose $X$ is of degree $d\ge 3$ and of dimension $n\ge 3$. Then given a line $\ell\subseteq X$, one of the following holds.
\begin{itemize}
\item[(i)] The normal bundle $\mathcal{N}_{\ell/X}\cong\mathcal{O}_{\ell}^{\oplus 2}\oplus \mathcal{O}_{\ell}(1)^{\oplus (n-3)}$ (the first kind); or 
\item[(ii)] The normal bundle $\mathcal{N}_{\ell/X}\cong\mathcal{O}_{\ell}(-1)\oplus \mathcal{O}_{\ell}(1)^{\oplus (n-2)}$ (the second kind).
\end{itemize} 
\end{lemma}

\begin{proof}
In our setting, the ample generator of $\textup{Pic}(X)$ is very ample (cf.~Theorem \ref{class-del-pezzo}) 
and hence it defines a closed embedding $X\hookrightarrow \mathbb{P}^N$. 
On the one hand, from the standard Euler sequence for a line $\ell$ on $X$, we have
$$0\to \mathcal{O}_\ell(2)\to T_X|_\ell\to \mathcal{N}_{\ell/X}\to 0.$$
By the splittingness of vector bundles over $\ell\cong\mathbb{P}^1$, we may assume that
$$\mathcal{N}_{\ell/X}=\mathcal{O}_\ell(a_1)\oplus\cdots\oplus\mathcal{O}_\ell(a_{n-1})$$
with $a_1\ge \cdots\ge a_{n-1}$ and $\sum_{i=1}^{n-1} a_i=n-3$.
On the other hand, from the relative normal bundle sequence
$$0\to\mathcal{N}_{\ell/X}\to\mathcal{N}_{\ell/\mathbb{P}^N}\to\mathcal{N}_{X/\mathbb{P}^N}\to 0,$$
each $a_i\le 1$.
In particular, there are only two possibilities, either $\mathcal{N}_{\ell/X}\cong\mathcal{O}_{\ell}^{\oplus 2}\oplus \mathcal{O}_{\ell}(1)^{\oplus n-3}$ or $\mathcal{N}_{\ell/X}\cong\mathcal{O}_{\ell}(-1)\oplus \mathcal{O}_{\ell}(1)^{\oplus n-2}$. 
\end{proof}

\begin{lemma}\label{lem-mukai-fourfold}
Let $X$ be a Mukai fourfold of Picard number 1.
Then $X$ is covered by lines. 
\end{lemma}

\begin{proof}
Let $\mathcal{K}$ be a dominating family of minimal rational curves. 
Then a general element $[C]$ of $\mathcal{K}$ is standard. 
Let $f:\mathbb{P}^1\rightarrow C$ be the normalization and we can write
$$
f^*T_X=\mathcal{O}_{\mathbb{P}^1}(2)\oplus \mathcal{O}_{\mathbb{P}^1}(1)^{\oplus p}\oplus \mathcal{O}_{\mathbb{P}^1}^{\oplus(3-p)}
$$
for some non-negative integer $p$. 
Then 
$$
p+2=-K_X\cdot C=2H\cdot C\leq 5,
$$
where $H$ is the ample generator of $\textup{Pic}(X)$. 
Hence, $H\cdot C=1$ or $2$. 
If $H\cdot C=1$, then $X$ is covered by lines, noting that a dominating family of minimal rational curves covers $X$. 

Therefore, we may assume that $H\cdot C=2$ and $p=2$. 
But then, since $\rho(X)=1$, by \cite[Theorem 0.1]{Miy04} (cf.~\cite{DH17}), $X$ is a smooth quadric of  index $4$, a contradiction. 
\end{proof}

With almost the same proof as in Lemma \ref{lem-lines-del}, we get the following result on the  normal bundles of lines on Fano manifolds of index 1.
\begin{lemma}\label{lem-normalbdle-index1}
Let $X$ be a Fano manifold of Picard number 1 and of index 1. 
If $X$ contains a line $\ell$, then the normal bundle $\mathcal{N}_{\ell/X}=\oplus\mathcal{O}(a_i)$ such that there is at least  one $a_i<0$.
\end{lemma}

To close this section, 
 we recall Lemmas \ref{lem-index-reduction} and \ref{lem-class-middleindex} below so as to deal with the case when  the index of a Fano manifold is large in Section \ref{sec-high-picard}. 

\begin{lemma}[{\cite[Theorem 2.2]{Wis94}; cf.~\cite[Theorems 7.2.1 and 7.2.2]{IP99}}]\label{lem-index-reduction}
Let $X$ be a Fano manifold of index $r$ and of dimension $n\ge 3$.
If $r>\frac{n}{2}+1$, then $\rho(X)=1$. 
If $r=\frac{n}{2}+1$ and $\rho(X)\ge 2$, then $X\cong\mathbb{P}^{r-1}\times \mathbb{P}^{r-1}$. 
If $r=\frac{n+1}{2}$ and $\rho(X)\ge 2$, then $X$ is  one of the following
$$\mathbb{P}_{\mathbb{P}^r}(\mathcal{O}(2)\oplus\mathcal{O}(1)^{\oplus(r-1)}),~~\mathbb{P}_{\mathbb{P}^r}(T_{\mathbb{P}^r}),~~\textup{or }~\mathbb{P}^{r-1}\times Q^r,$$
where $T_{\mathbb{P}^r}$ is the tangent bundle of $\mathbb{P}^r$ and $Q^r$ is a smooth quadric of dimension $r$.
\end{lemma}

\begin{lemma}[{cf.~e.g.~\cite[(3.1), (4.1), (5.1), (6.3)]{Wis94}}]\label{lem-class-middleindex}
Let $X$ be a Fano manifold of  $\dim(X)=2r\ge 6$, $i(X)=r$ (i.e., $-K_X\sim rL$ for some ample Cartier divisor $L$), and $\rho(X)\ge 2$. 
Then either (i) $X\cong\mathbb{P}^2\times\mathbb{P}^2\times\mathbb{P}^2$, or (ii) $\rho(X)=2$ and there exists a contraction $p:X\to Y$ of an extremal ray such that $\dim (Y)<\dim (X)$ and all fibres of $p$ are of dimension $\le r$.
Further, if (ii) happens,  one of the following occurs (with $\mathcal{E}:=p_*L$).
\begin{enumerate}
\item $\dim (Y)=r+1$, $\mathcal{E}$ is  locally free of rank $r$, $X=\mathbb{P}_Y(\mathcal{E})$  is one of the following
$$V_d\times\mathbb{P}^{r-1},~~\mathbb{P}_{\mathbb{P}^{r+1}}(\mathcal{O}(2)^{\oplus 2}\oplus\mathcal{O}(1)^{\oplus(r-2)}),~~\mathbb{P}_{\mathbb{P}^{r+1}}(\mathcal{O}(3)\oplus\mathcal{O}(1)^{\oplus(r-1)}),$$
$$\mathbb{P}_{Q^{r+1}}(\mathcal{O}(2)\oplus\mathcal{O}(1)^{\oplus (r-1)}),~~\mathbb{P}_{Q^4}(\mathbb{E}(1)\oplus\mathcal{O}(1))$$
where $V_d$ is a del Pezzo manifold of degree $d$, $Q^n$ is a smooth quadric of dimension $n$, and $\mathbb{E}$ denotes a spinor bundle on a 4-dimensional smooth quadric.
\item $\dim (Y)=r$, $\mathcal{E}$ is  locally free  of rank $r+2$, and $p$ is a quadric bundle so that $X$  embeds into $\mathbb{P}(\mathcal{E})$ as a divisor of relative degree 2. 
\item $\dim (Y)=r+1$, $\mathcal{E}$ is  reflexive  of rank $r$ 
and $X\cong\mathbb{P}(\mathcal{E})$.
Further, if $X$ is not of type (2), then $X$ is either a blow-up of  a smooth quadric $Q^{2r}$ along a linear space $\mathbb{P}^{r-1}$, or an intersection of two smooth divisors in $\mathbb{P}^{r+1}\times\mathbb{P}^{r+1}$ of bidegree $(1,1)$.
\end{enumerate}
\end{lemma}

\section{Morphisms onto weighted hypersurfaces, Proof of Theorem \ref{main-del-pezzo-non}}
In this section, we consider Conjecture \ref{main-conj-bounded} (a stronger version of Question \ref{ques-main-bounded-fano})  when $X$ is a weighted hypersurface and prove Theorem \ref{main-del-pezzo-non}.
By Theorem \ref{class-del-pezzo}, del Pezzo manifolds of lower degrees are hypersurfaces in some weighted projective spaces.
Therefore, the following theorem plays a significant role in the proof of Theorem \ref{main-del-pezzo-non}.

\begin{theorem}\label{thm-wps-ci}
Let $X$ be a smooth weighted hypersurface of dimension $n\ge 2$ and of degree $d$ in a  weighted projective space $\mathbb{P}(a_0,a_1,\cdots,a_{n+1})$ with $a_0\ge a_1\ge a_i=1$ for each $i\ge 2$. 
Suppose $a_0$ and $a_1$ are coprime and $d\ge a_0+a_1+1$.  
Then $c_n(\Omega_X(a_0+a_1))>\mathcal{O}_X(1)^n(a_0+a_1)^n$.
Further, for any  smooth projective variety $Y$ of dimension $n$ and of Picard number 1, there is a number $N>0$ such that the degree of any finite morphism $Y\to X$ is less than $N$; thus, $X$ admits no non-isomorphic endomorphism. 
\end{theorem}

\begin{remark}\label{rem-locally-free-wps}
By Notation \ref{notation2.1}, in a weighted projective space $\mathbb{P}=\mathbb{P}(a_0,\cdots,a_{n+1})$, if $a_i$ are mutually coprime, then $\mathbb{P}$ has only isolated singularities.
In this case,  every smooth weighted hypersurface of dimension $\ge 2$ is well-formed and thus contained in the smooth locus  of $\mathbb{P}$. 
So $\mathcal{O}_X(n)$ is locally free and $\mathcal{O}_X(a+b)=\mathcal{O}_X(a)\otimes\mathcal{O}_X(b)$ for any $n,a,b\in\mathbb{Z}$. 
\end{remark}

The following corollary on the smooth hypersurfaces in the usual projective space is an immediate consequence of Theorem \ref{thm-wps-ci} and \cite{Ame07} (cf.~\cite{Bea01}).
\begin{corollary}\label{cor-usual-hyper}
Let $X\subseteq\mathbb{P}^{n+1}$ be a smooth hypersurface of degree $d$.
If either $d\ge 3$, or $d=2$ and $n\ge 3$, then for any smooth projective variety $Y$ of dimension $n$ and of Picard number 1, there is a positive number $N$ such that any finite morphism $Y\to X$ has the degree no more than $N$.
\end{corollary}

We do some preparations before proving Theorem \ref{thm-wps-ci}.

\begin{lemma}\label{lem-wps-gg}
Let $X$ be a smooth well-formed weighted hypersurface of dimension $n\ge 2$ and of degree $d$ in a  weighted projective space $\mathbb{P}=\mathbb{P}(a_0,a_1,\cdots,a_{n+1})$ with $a_0\ge a_1\ge a_i=1$ for each $i\ge 2$.  
If $a_1=1$ (resp.~$a_1>1$), then $\Omega_{X}(a_{0}+a_{1})$ is globally generated (resp.~globally generated away from finitely many points). 
\end{lemma}
\begin{proof}
Let $x_i$ be the coordinates of $\mathbb{P}$ with $\deg(x_i)=a_i$. We may assume that $a_0>1$; otherwise, $\mathbb{P}$ is the usual projective space and our lemma follows from \cite{Bea01}.  
Denote by $\mathbb{P}_0$ the smooth locus of $\mathbb{P}$. 
By Remark \ref{rem-locally-free-wps}, $X\subseteq\mathbb{P}_0$.
Suppose that $a_1>1$ (resp. $a_1=1$). 
Then the reflexive sheaf $\mathcal{O}_{\mathbb{P}}(1)$ is not globally generated precisely along the line $\ell$ (resp. the vertex $p$) which is defined by  $x_2=\cdots=x_{n+1}=0$ (resp. $x_1=\cdots=x_{n+1}=0$). From the generalized (weighted) Euler short exact sequence, we have 
$$\bigoplus_{i=0}^{n+1}\mathcal{O}_{\mathbb{P}_0}(a_i)\to T_{\mathbb{P}_0}\to 0.$$
Then we have
$$\bigwedge^n\bigoplus_{i=0}^{n+1}\mathcal{O}_{\mathbb{P}_0}(a_i)\to \bigwedge^nT_{\mathbb{P}_0}\to 0$$
By the adjunction formula (cf.~\cite[Chapter II, Exercise 5.16 (b)]{Har77} and \cite[Chapter V, Exercise 1.3.5]{Kol96}), we have $\bigwedge^nT_{\mathbb{P}_0}=\Omega_{\mathbb{P}_0}(\sum a_i)$. 
Since 
$$\bigwedge^n\bigoplus_{i=0}^{n+1}\mathcal{O}_{\mathbb{P}_0}(a_i)=\bigoplus_{0\le k\neq l\le n+1}\mathcal{O}_{\mathbb{P}_0}((\sum_{i=0}^{n+1} a_i)-a_k-a_l),$$
after twisting the invertible sheaf $\mathcal{O}_{\mathbb{P}_0}(a_0+a_1)$  (cf.~Remark \ref{rem-locally-free-wps}),  there is a surjection 
$$\bigoplus_{0\le k\neq l\le n+1}\mathcal{O}_{\mathbb{P}_0}((\sum_{i=0}^{n+1} a_i)+a_0+a_1-a_k-a_l)\twoheadrightarrow\Omega_{\mathbb{P}_0}((\sum_{i=0}^{n+1} a_i)+a_{0}+a_{1}).$$
In particular, $\Omega_{\mathbb{P}_0}(a_{0}+a_{1})$ is globally generated outside the line $\ell$ (resp. the vertex $p$).
Now that $X\subseteq\mathbb{P}_0$, from the standard exact sequence
$$0\to \mathcal{O}_X(-d) \to \Omega_{\mathbb{P}_0}|_X\to \Omega_X\to 0,$$
there is a surjection $\Omega_{\mathbb{P}_0}(a_{0}+a_{1})|_X\to \Omega_X(a_{0}+a_{1})$.
Note that $X$ intersects with  $\ell$ along finitely many points (resp. $X$ is disjoint with the vertex $p$). 
In particular, $\Omega_X(a_0+a_1)$ is globally generated away from finitely many points (resp. $\Omega_X(a_0+a_1)$ is globally generated).
The lemma is thus proved.
\end{proof}

\begin{proposition}\label{pro-top-chern-cal}
Let $X$ be a smooth well-formed weighted hypersurface of dimension $n\ge 2$ and of degree $d$ in a  weighted projective space $\mathbb{P}=\mathbb{P}(a_0,a_1,\cdots,a_{n+1})$ with $a_0\ge a_1\ge a_i=1$ for each $i\ge 2$.   
Suppose that $d\neq a:=a_0+a_1$. 
Then we have 
$$c_n(\Omega_X(a))=\mathcal{O}_X(1)^n\cdot \frac{d\prod_{i=0}^{n+1}(a-a_i)-a\prod_{i=0}^{n+1}(d-a_i)+(-1)^{n+1}(d-a)\prod_{i=0}^{n+1}a_i}{ad(a-d)}.$$
\end{proposition}

\begin{proof}
Denote by $h=\mathcal{O}_X(1)$ the hyperplane section of $X$.
Notice that $X$ is contained in the smooth locus $\mathbb{P}_0$ of $\mathbb{P}$. 
Twisting the generalized Euler sequence and the conormal sequence by the locally free sheaf $\mathcal{O}_X(a)$, 
we have
\begin{align*}
0\to\Omega_{\mathbb{P}_0}(a)|_X\to \bigoplus_{i=0}^{n+1}\mathcal{O}_X(-a_i+a)\to\mathcal{O}_X(a)\to 0,\\
0\to \mathcal{O}_X(a-d)\to \Omega_{\mathbb{P}_0}^1(a)|_X\to\Omega_X(a)\to 0.
\end{align*} 
Then the Chern class 
$$c(\Omega_X(a))=\frac{\prod_{i=0}^{n+1}(1+(a-a_i)h)}{(1+(a-d)h)(1+ah)}.$$
Hence, the top Chern class $c_n(\Omega_X(a))$ is the product of $h^n$ and its coefficient, which is the residue of the following form
$$\omega=\frac{\prod_{i=0}^{n+1}(1+(a-a_i)x)}{(1+(a-d)x)(1+ax)x^{n+1}}dx$$
at the origin $x=0$.
By the residue theorem, we have
$$(*)~~\textup{Res}_{\infty}\omega+\textup{Res}_{\frac{1}{d-a}}\omega+\textup{Res}_{-1/a}\omega+\textup{Res}_0\omega=0.$$
Clearly, we have
$$\textup{Res}_{\infty}\omega
=\frac{\prod_{i=0}^{n+1}(a-a_i)}{(d-a)a},~~
\textup{Res}_{\frac{1}{d-a}}\omega=\frac{\prod_{i=0}^{n+1}(d-a_i)}{(a-d)d},~~\textup{Res}_{-1/a}\omega=\frac{(-1)^{n+1}\prod_{i=0}^{n+1}a_i}{ad}.$$
Therefore, applying $(*)$, we get the desired equality
$$c_n(\Omega_X(a))=h^n\cdot \frac{d\prod_{i=0}^{n+1}(a-a_i)-a\prod_{i=0}^{n+1}(d-a_i)+(-1)^{n+1}(d-a)\prod_{i=0}^{n+1}a_i}{ad(a-d)}.$$
We finish the proof of our proposition.
\end{proof}

\begin{proof}[Proof of Theorem \ref{thm-wps-ci}] 
Let $a:=a_0+a_1$. 
By Proposition \ref{pro-top-chern-cal}, we have
\begin{align*}
&c_n(\Omega_X(a))-\mathcal{O}_X(1)^n\cdot a^n\\
&=\mathcal{O}_X(1)^n((d-1)^n-a^n)+\frac{\mathcal{O}_X(1)^na_0a_1}{ad(d-a)} (a(d-1)^n-d(a-1)^n+(-1)^n(d-a))\\
&\ge \frac{\mathcal{O}_X(1)^na_0a_1}{ad(d-a)} (a(d-1)^n-d(a-1)^n-(d-a)).
\end{align*}
Let $g(x)=a(x-1)^n-x(a-1)^n-(x-a)$. 
It is easy to show that $g(x)$ is monotonically increasing when $x\ge a+1$. 
Therefore, 
\begin{align*}
   g(d)\ge g(a+1)=a^{n+1}-(a+1)(a-1)^n-1>a^2(a^{n-1}-(a-1)^{n-1})-1 >0.
\end{align*}
So the first part of our theorem follows.  
The second part of our theorem follows from Lemmas \ref{lem-wps-gg} and \ref{lem-arv-prop2.1}. 
We finish the proof of our theorem.
\end{proof}

As an application of Theorem \ref{thm-wps-ci}, we shall prove Theorem \ref{main-del-pezzo-non}.

Let us first introduce the following proposition which confirms Conjecture \ref{main-conj-bounded} for a smooth complete intersection of two quadrics, i.e., a del Pezzo manifold of degree 4 (cf.~Theorem \ref{class-del-pezzo}). 
We note that the same strategy can be applied to the smooth complete intersection of more than two quadrics if their defining equations are ``diagonalizable''.

\begin{proposition}\label{pro-complete-intersection-bounded}
Let $X$ be a del Pezzo manifold of degree 4 and of dimension $n\ge 3$.
For any smooth projective variety $Y$ of dimension $n$ and of Picard number 1, there is a  number $N>0$ such that the degree of any finite morphism $Y\to X$
is no more than $N$.
\end{proposition}

\begin{proof}
After a suitable change of basis, we may assume that $X\subseteq\mathbb{P}^{n+2}$ is given by the complete intersection of two smooth quadrics $Q_1$ and $Q_2$ of dimension $n+1$
$$Q_1:\sum_{i=0}^{n+2}x_i^2=0,~~~Q_2:\sum_{i=0}^{n+2}\lambda_ix_i^2=0$$
with $\lambda_i\neq\lambda_j$ if $i\neq j$ (cf.~\cite[Proposition 2.1]{Rei72}). 
Then the projection along the $i$-th coordinate is a double cover
$$\pi_i: X\to Q^n$$
where $Q^n\subseteq\mathbb{P}^{n+1}$ is a smooth quadric with the defining equation $\sum_{j\neq i}(\lambda_j-\lambda_i)x_j^2=0$. 
Then any finite surjective morphism $Y\to X$ induces a finite surjective morphism $Y\to Q^n$ and hence our proposition follows from \cite{Ame07}. 
\end{proof}

\begin{remark}
 We give a remark of the proof of Proposition \ref{pro-complete-intersection-bounded} in terms of the branch locus. 
  Let us fix the projection along the \(0\)-th coordinate \(\pi=\pi_0:X\to Q^n\).
  Then it is a double cover which is totally ramified along the prime divisor \(R_{\pi}=X\cap\{x_0=0\}\).
  The branch divisor  \(B_{\pi}=Q^n\cap \{\sum_{j=1}^{n+2}x_j^2=0\}\) is a degree 2 hypersurface on \(Q^n\) and we have the ramification divisor formula \(K_X=\pi^*(K_{Q^n}+\frac{1}{2}B_{\pi})\).  
\end{remark}

\begin{proof}[Proof of Theorem \ref{main-del-pezzo-non}]
Let $H$ be the ample generator of $\textup{Pic}(X)$. 
In the view of \cite{ARV99} or \cite{HM03}, we may assume that $n\ge 4$.
We  apply Theorem \ref{class-del-pezzo}. 
Since $\rho(X)=1$, we have $d:=H^n\le 5$. 
If $H^n=5$, then $n\le 6$ and it is known that $X$ is homogeneous ($n=6$) or almost homogeneous  ($n=4$ or $5$); see \cite[Proposition 2.3]{FM19}, in particular, our theorem follows from \cite[Proposition 2]{PS89} and \cite[Theorem 1.4]{HN11}. 
If $H^n=4$, then  our theorem follows from Proposition \ref{pro-complete-intersection-bounded}.
If $H^n=3$ (resp. 2, 1), then $X$ is a smooth weighted hypersurface of degree 3 (resp. 4, 6) in $\mathbb{P}(1^{\oplus (n+2)})$ (resp. $\mathbb{P}(2,1^{\oplus (n+1)})$, $\mathbb{P}(3, 2,1^{\oplus n})$), and Theorem \ref{thm-wps-ci} implies our result.
So Theorem \ref{main-del-pezzo-non} is proved.
\end{proof}

We end up this section with the following remark on Mukai manifolds (cf.~\cite{Muk89} and \cite[Chapter 5 and Corollary 2.1.17]{IP99}).
\begin{remark}[A first glance at Mukai manifolds]\label{rem-mukai-rho=1}
One can continue to study  Conjecture \ref{main-conj-pn} according to the index of the Fano manifold $X$.
The next step is to consider the Mukai manifolds, i.e.,  the index  is $\dim(X)-2$.
In this case, the genus $g=g(X)$  plays a significant role (cf.~Notation \ref{notation2.1}). 
By the classification of Mukai, it has been shown that $2\le g\le 12$ and $g\neq 11$. 
From the viewpoints of the authors, when  the genus is small,  the tangent bundle of $X$ does not have strong positivity, in which case, similar strategies  like Theorem \ref{thm-wps-ci} and Proposition \ref{pro-complete-intersection-bounded} in this section  work. 
For example, if $g=4$, then $X$ is a complete intersection of a quadric and a cubic in $\mathbb{P}^{n+2}$ and then a generalized version of Theorem \ref{thm-wps-ci} shows that Conjecture \ref{main-conj-bounded} (and hence Conjecture \ref{main-conj-pn}) holds for such $X$.
If $g=3$, then either $X$ is a quartic hypersurface in $\mathbb{P}^{n+1}$ or $X$ is a double cover  over a smooth quadric $Q$ (of dimension $\ge 3$); in the former (resp. latter) case, Theorem \ref{thm-wps-ci} (resp. Proposition \ref{pro-complete-intersection-bounded})  concludes that Conjecture \ref{main-conj-bounded} (and hence Conjecture \ref{main-conj-pn}) is true for $X$. 
If $g=2$, then $X$ is a degree 6 hypersurface in the weighted projective space $\mathbb{P}(3,1^{\oplus(n+1)})$, and then  Theorem \ref{thm-wps-ci}  implies Conjecture \ref{main-conj-bounded} (and hence Conjecture \ref{main-conj-pn}) is true in this case.  
However, when  $g\ge 6$,  the tangent bundles are expected to be a bit positive.
Indeed, in most of the cases $6\le g\le 12$, $X$ is a linear section of some rational homogeneous variety, in which circumstance, Theorem \ref{thm-wps-ci} seems not to work any more.
\end{remark}

\section{Boundedness for morphisms onto del Pezzo fourfolds of degree 5}
In this section, we verify Conjecture \ref{main-conj-bounded} for del Pezzo fourfolds of degree 5. 
The following proposition is our main result of this section. 
\begin{proposition}[{cf.~\cite[Proposition 2.2]{Ame97}}]\label{pro-quintic-4fold}
Let $X$ be a del Pezzo fourfold of degree 5.   
Then for any smooth projective fourfold $Y$ with $b_2(Y)=1$, there is a positive number $N$ such that any finite morphism $Y\to X$ has the degree no more than $N$. 
\end{proposition}

We do some preparations before proving Proposition \ref{pro-quintic-4fold}. 
First, the result below is  borrowed from \cite{Pro94} and \cite{Chu23} (cf.~Lemma \ref{lem-lines-del}).

\begin{proposition}[{\cite[Proposition 2.2]{Pro94}, \cite[Proof of Corollary 3.2]{Chu23}}]\label{del-pezzo-fourfold-deg5-negative}
Let $X$ be a del Pezzo fourfold of degree 5.
Then the  lines of the second kind cover a  divisor on $X$.
\end{proposition}

\begin{proof}
By \cite[Proposition 2.2]{Pro94},  $X$ contains exactly one plane $S$ which is a Schubert variety of type $\sigma_{2,2}$, and a 1-parameter family of planes $P_t$ that are Schubert varieties of type $\sigma_{3,1}$.
Besides, each  $P_t$ intersects the unique $S$ along a line tangent to a fixed  conic $C\subseteq S$.
In addition, the union of all $\sigma_{3,1}$-planes $P_t$ on $X$ is a singular hyperplane section $R$. 
Moreover, if $\ell\subseteq R$ is a line, then $\ell\cap S\neq\emptyset$ and either $\ell$ is contained in $S$ or $\ell$ is contained in some $P_t$. 
By \cite[Corollary 3.2]{Chu23}, a line  is of the second kind if and only if it intersects  the conic $C$ along a single point.
In particular, every $P_t$  is covered by the lines of the second kind. 
So the lines of the second kind covers the prime divisor $R$.
\end{proof}

The following Lemmas \ref{lem-tangent-singular} and \ref{lem-bertini} are well-known to experts. 
\begin{lemma}\label{lem-tangent-singular}
Let $f:Y\to X\subseteq \mathbb{P}^n$ be a finite morphism of non-singular quasi-projective varieties and $H$ a hyperplane  not containing $X$. 
Then $f^{-1}(X\cap H)$ is non-singular at a point $y\in Y$ if and only if $f_*(T_{Y,y})$ is not contained in $H$.
\end{lemma}

\begin{proof}
Let $X_1:=X\cap H$ and $Y_1:=f^{-1}(X_1)$.
Then $Y_1$ is a hypersurface defined by $Z(f^*(s))$ where $s\in|\mathcal{O}_X(1)|$ and $X_1=Z(s)$. 
Then $Y_1$ is smooth along $y\in Y_1$ if and only if $\mathcal{O}_{Y_1,y}=\mathcal{O}_{Y,y}/f^*(s)_y$ is a regular local ring if and only if $f^*(s)_y\not\in\mathfrak{m}_y^2$. 
Let $x=f(y)$ be the image. 
Since $f$ induces a natural cotangent map $f^*:\mathfrak{m}_x/\mathfrak{m}_x^2\to \mathfrak{m}_y/\mathfrak{m}_y^2$,
$f^*(s_x)=0$ in $\mathfrak{m}_y/\mathfrak{m}_y^2$ if and only if for any tangent vector $v\in T_{Y,y}=(\mathfrak{m}_y/\mathfrak{m}_y^2)^{\vee}$, we have $v(f^*(s_x))=(f_*v)(s_x)=0$.
But this is equivalent to $f_*T_{Y,y}\subseteq Z(s_x)=T_{X,x}\cap H$.
So our lemma is proved.
\end{proof}

\begin{lemma}\label{lem-bertini}
Let $X\subseteq\mathbb{P}^N$ be a smooth projective variety of dimension $n\ge 3$. 
Assume that $X$ contains a line  $\ell$.
Then a general hyperplane section of $X$ containing $\ell$ is smooth.
\end{lemma}

\begin{proof}
Denote by $V$ the sub-linear system of $|\mathcal{O}_{\mathbb{P}^N}(1)|$, each element of which contains $\ell$. 
Then $\dim(V)=N-2$.
By Bertini's theorem,  $H\cap X$ is smooth outside $\ell$ for a general  $H\in V$. 
Let $B\subseteq \ell\times |\mathcal{O}_{\mathbb{P}^N}(1)|$ be the Zariski closed subset consisting of all pairs $(x,H)$ such that $H\cap X$ is singular at some closed point $x\in\ell$.  Let $B_x$ be the fibre of the natural projection $B\rightarrow \ell$ at the point $x\in \ell$.
Then $\dim(B)\le \dim(\ell)+\dim(B_x)=N-n< \dim(V)$ (cf.~\cite[Chapter II, Proof of Theorem 8.18]{Har77}).
In particular, we can pick a general hyperplane $H\in V$  such that $H\cap X$ is smooth.
\end{proof}

\begin{proof}[Proof of Proposition \ref{pro-quintic-4fold}]
Let $f:Y\to X$ be a finite surjective morphism from a smooth projective fourfold with $b_2(Y)=1$. 
Let $A:=\{y\in Y\,|\,\textup{rank}\,df_y\le 2\}$ and $B:=f(A)$.
Since $f$ is finite, by \cite[Chapter III, Proposition 10.6]{Har77}, we have $\dim(A)=\dim(B)\le 2$.
By Proposition \ref{del-pezzo-fourfold-deg5-negative}, the lines of the second type cover a prime divisor. 
So there is a second type line $\ell$ on $X$ such that every irreducible component of $f^{-1}(\ell)$ is not contained in $A$. 

We claim that a general linear complete intersection surface $S_1$ containing $\ell$ is a  del Pezzo surface of degree 5 such that $S_2:=f^{-1}(S_1)$ has only isolated singularities. 
The smoothness of $S_1$ follows from Lemma \ref{lem-bertini}.
Moreover, by the choice of $\ell$, we know  $\dim (f_*(T_{Y,y}))\ge 3$ for a  general point $y\in f^{-1}(\ell)$. 
Taking a general hyperplane $H_1\subseteq \mathbb{P}^7$ containing $\ell$, we see that  $f_*T_{Y,y}\not\subseteq H_1$. 
In particular, by Lemma \ref{lem-tangent-singular}, $Y_1:=f^{-1}(X_1:=X\cap H_1)$ has only isolated singularities. 
Consider the following  sequence
$$f^*\Omega_X\otimes k(y)\to\Omega_Y\otimes k(y)\to\Omega_{Y/X}\otimes k(y)\to 0.$$
Taking the dual map with $x=f(y)$ where $y\in f^{-1}(\ell)$ is a general point, we have
$$0\to (\Omega_{Y/X}\otimes k(y))^{\vee}\to (\mathfrak{m}_y/\mathfrak{m}_y^2)^{\vee}\to (\mathfrak{m}_x/\mathfrak{m}_x^2)^{\vee}.$$
Here, we use the isomorphisms $\mathfrak{m}_x/\mathfrak{m}_x^2\cong\Omega_{X}\otimes k(x)$ and $\mathfrak{m}_y/\mathfrak{m}_y^2\cong\Omega_{Y}\otimes k(y)$  (cf.~\cite[Chapter II, Proposition 8.7]{Har77}).  
Then $\dim ((\Omega_{Y/X}\otimes k(y))^{\vee})=\dim(T_{Y,y})-\dim(f_*T_{Y,y}))\le 1$. 
Taking the base change $X_1\subseteq X$, we have $\dim ((\Omega_{Y_1/X_1}\otimes k(y))^{\vee})\le 1$ (cf.~\cite[Chapter II, Proposition 8.10]{Har77}).
With the same reason as above, we have
$\dim(T_{Y_1,y})-\dim(f_*T_{Y_1,y}))\le 1$; in particular, $\dim (f_*T_{Y_1, y})\ge 2$. 
Let $Y_1^\circ$ be the open regular locus of $Y_1$, and $f^\circ:Y_1^\circ\to X_1$ the induced map. 
Let $A_1:=\{y\in Y_1^\circ\,|\,\textup{rank}\,df^\circ_y\le 1\}$ and $B_1=\overline{f^\circ(A_1)}$.
Then $\dim(A_1)=\dim(B_1)\le 1$ and thus $\ell\not\subseteq B_1$.
Hence, one can pick another hyperplane $H_2$ such that $f_*T_{(f^\circ)^{-1}(S_1:=X_1\cap H_2),y}\not\subseteq H_2\cap H_1$ for a general $y\in (f^\circ)^{-1}(\ell)$. 
By Lemma \ref{lem-tangent-singular}, $S_2:=f^{-1}(S_1)$ has only isolated singularities  and our claim is proved.

Let $C=f^*(\ell)$ be the scheme-theoretic inverse image defined by the fibre product. 
Consider the relative normal bundle sequence
$$(*)~0\to\mathcal{N}_{\ell/S_1}\to\mathcal{N}_{\ell/X}\to\mathcal{N}_{S_1/X}|_{\ell}\to 0.$$
Clearly, $\mathcal{N}_{\ell/X}=\mathcal{O}_{\ell}(1)^{\oplus 2}\oplus\mathcal{O}_{\ell}(-1)$ (cf.~Lemma \ref{lem-lines-del}) and $\mathcal{N}_{S_1/X}|_{\ell}=\mathcal{O}_\ell(1)^{\oplus 2}$ by the choice of $S_1$.
Then $\mathcal{N}_{\ell/S_1}=\mathcal{O}_\ell(-1)$ and the sequence $(*)$ splits.
Pulling back $(*)$ to $Y$, we have 
$$(**)~0\to \mathcal{N}_{C/S_2}\to \mathcal{N}_{C/Y}\to \mathcal{N}_{S_2/Y}|_C\to 0$$
which also splits.
By the  choice of $S_1$, its pullback $S_2$ is also a complete intersection $f^*(H_1\cap X)\cap f^*(H_2\cap X)$ and hence Cohen-Macaulay. 
Since $S_2$ is regular in codimension one, $S_2$ is normal with $C$ being a  Cartier divisor on $S_2$, noting that so is $\ell$ on  $S_1$. 
By \cite[Lemma 2.3]{Bra92} (cf.~\cite[Claim in Proof of Proposition 2.2]{Ame97}), the sequence 
$(**)$ splits if and only if $C$ is the restriction of an effective Cartier divisor from the second infinitesimal neighborhood $S_2^{(2)}$ of $S_2$ in $X$. 

\begin{claim}\label{claim_1-dim-del-pezzo}
The image of the map $\textup{Pic}(S_2^{(2)})\to \textup{N}^1(S_2)$ is 1-dimensional if $\deg(f)\gg 1$.
\end{claim}

Suppose Claim \ref{claim_1-dim-del-pezzo} for the time being.
By the proof of \cite[\S 2.4]{Bra92},  
$C\equiv \lambda H_Y|_{S_2}$ for some $\lambda>0$, where $H_Y$ is the ample generator of $\textup{N}^1(Y)$.
However, this is absurd, since 
$$0>\deg\mathcal{N}_{C/S_2}=(C^2)_{S_2}=(C\cdot \lambda H_Y|_{S_2})_{S_2}=\lambda (C\cdot H_Y)_Y>0.$$

So we are left to show  Claim \ref{claim_1-dim-del-pezzo}.
By \cite[Proof of Lemma 1.2]{Bra92}, the natural map
$$\textup{Pic}(S_2^{(2)})\to \textup{Pic}(S_2)\to\textup{N}^1(S_2)\to H^1(S_2,\Omega_{S_2}^1)$$
factors through
$$\textup{Pic}(S_2^{(2)})\to H^1(S_2^{(2)},\Omega_{S_2^{(2)}}^1)\to H^1(S_2,\Omega_{S_2^{(2)}}^1|_{S_2})\to H^1(S_2,\Omega_Y^1|_{S_2})\to H^1(S_2,\Omega_{S_2}^1).$$
Here,  $\Omega_{S_2^{(2)}}|_{S_2}\cong\Omega_Y^1|_{S_2}$ since the embedding $S_2\to Y$ factors through a thickening $S_2\to S_2^{(2)}\to Y$ (cf.~e.g.~\cite[Proposition 8.12]{Har77}). 
In addition, it is known that the natural map $\textup{N}^1(S_2)\to H^1(S_2,\Omega_{S_2}^1)$ is an injection.
So we only need to show $h^1(S_2,\Omega_Y^1|_{S_2})=1$ if  $\deg(f)\gg 1$. 
Write $f^*H_X\equiv mH_Y$ for some positive integer $m$ where $H_X$ and $H_Y$ are the ample generators of $\textup{N}^1(X)$ and $\textup{N}^1(Y)$ respectively.
Then $S_2=H_1'\cdot H_2'$ where $H_i'\in|\mathcal{O}_Y(mH_Y)|$. By the standard sequence of complete intersection,  
$$0\to\mathcal{O}_Y(-2mH_Y)\to \mathcal{O}_Y(-mH_Y)\oplus\mathcal{O}_Y(-mH_Y)\to\mathcal{I}_{S_2}\to 0$$
where $\mathcal{I}_{S_2}$ is the ideal sheaf of $S_2$.
Twisting it with $\Omega_Y$, we have
$$0\to\Omega_Y(-2mH_Y)\to \Omega_Y(-mH_Y)\oplus\Omega_Y(-mH_Y)\to\Omega_Y(\mathcal{I}_{S_2})\to 0.$$
By \cite[Chapter III, Theorem 7.6]{Har77}, there is an integer $m_0$ such that $H^i(Y,\Omega_Y(\mathcal{I}_{S_2}))=0$ for all $i<3$ whenever $m\ge m_0$.
On the other hand, we have the natural sequence
$$0\to \Omega_Y(\mathcal{I}_{S_2})\to \Omega_Y\to \Omega_Y|_{S_2}\to 0.$$
Hence, by the vanishing $H^i(Y,\Omega_Y(\mathcal{I}_{S_2}))=0$ for $i=1,2$, we have
$$H^1(S_2,\Omega_Y|_{S_2})=H^1(Y,\Omega_Y)=\mathbb{C},$$
noting that $1\le h^{1,1}(Y)\le b_2(Y)=1$. 
So we complete the proof of our claim. 
\end{proof}

The corollary below follows  from Proposition \ref{pro-quintic-4fold} and proofs of Theorems \ref{main-del-pezzo-non} and \ref{class-del-pezzo}.

\begin{corollary}\label{coro-del-pezzo-bounded}
Let $X$ be a del Pezzo manifold of Picard number 1 and of dimension $n$.
If $X$ is not  of degree 5 with $n=5$ or $6$, then for every smooth projective variety $Y$ of dimension $n$ and of $b_2(Y)=1$, there is a positive number $N$ such that every finite morphism $Y\to X$ has the degree no more than $N$.
\end{corollary}

\section{Proof of Theorem \ref{rho=1-dim=4}}
This short section is devoted to the proof of Theorem \ref{rho=1-dim=4}. 
Let us begin with a special case when the index of a Fano manifold is 2.
\begin{proposition}\label{pro-index=2}
Let $X$ be a Fano manifold of  index 2 and of Picard number 1.
If $X$ is covered by lines, then for any smooth Fano fourfold $Y$ of Picard number 1, there is a  number $N>0$ such that any finite  morphism $Y\to X$ 
has the degree no more than $N$.
\end{proposition}

\begin{proof}
Let $n=\dim(X)$. 
Let $\mathcal{K}$ be a dominating family of minimal rational curves in which the general elements are lines by our assumption. 
For a general line $C\in \mathcal{K}$, if $f:\mathbb{P}^1\rightarrow C$ is the normalization, then we have 
$$
f^*T_X=\mathcal{O}_{\mathbb{P}^1}(2)\oplus \mathcal{O}_{\mathbb{P}^1}(1)^{\oplus p}\oplus \mathcal{O}_{\mathbb{P}^1}^{\oplus(n-p-1)}
$$
for some  integer $p\ge 0$.
Since  the index $i(X)=2$, we see that $p=0$; in particular, the variety of minimal rational tangents $\mathcal{C}_x$ along a general point $x\in X$ is $0$-dimensional.
So our proposition follows from \cite[Theorem 2]{HM03}. 
\end{proof}

\begin{corollary}\label{mukai-pic=1-dim=4}
Let $X$ be a Mukai fourfold of Picard number 1. 
Then for any smooth Fano fourfold $Y$ of Picard number 1, there is a positive number $N$ such that every finite  morphism $Y\to X$ has the degree no more than $N$.
\end{corollary}
\begin{proof}
It is an immediate consequence of Proposition \ref{pro-index=2} and Lemma \ref{lem-mukai-fourfold}.
\end{proof}

\begin{proof}[Proof of Theorem \ref{rho=1-dim=4}]
In the view of Corollary \ref{mukai-pic=1-dim=4}, we may assume that the index is 3.
But then such a del Pezzo fourfold satisfies the boundedness property by Proposition \ref{pro-quintic-4fold} and the proof of Theorem \ref{main-del-pezzo-non} (cf.~Theorems  \ref{thm-wps-ci} and \ref{class-del-pezzo}).
\end{proof}

\begin{remark}\label{rem-index=1-dim=4}
Let $X$ be a smooth Fano fourfold of Picard number 1 and of  index 1.
Let $\mathcal{K}$ be a dominating family of minimal rational curves.
Given a standard rational curve $\ell$ of $\mathcal{K}$, we have
$$T_X|_\ell\cong\mathcal{O}_\ell(2)\oplus\mathcal{O}_\ell(1)^{\oplus p}\oplus\mathcal{O}_\ell^{\oplus (3-p)}.$$
Then $p=0,1,2$ or $3$.  
If $p=0$, then $\ell$ is a smooth rational curve with trivial normal bundle; in particular, Question \ref{ques-main-bounded-fano} has a positive answer by \cite[Theorem 2]{HM03}.
If $p=2$, then $-K_X\cdot\ell=4$ and hence $X$ is a smooth quadric by \cite{Miy04} (cf.~\cite{DH17}), which is impossible.
If $p=3$, then $-K_X\cdot\ell=5$; in particular, $X\cong\mathbb{P}^4$  which cannot happen, either.
So  we are left to consider the case $p=1$, which is Question \ref{ques-fourfold-bounded}.

On the other hand, to confirm Conjecture \ref{main-conj-bounded} for Mukai fourfolds, one  needs to verify the case when the genus $g\ge 5$ (cf.~Remark \ref{rem-mukai-rho=1}).
This is one of our forthcoming work.
\end{remark}

\section{Endomorphisms of quadrics, Proof of Theorem \ref{thm-singular-quadric}} 
In this section, we shall classify the  singular quadric hypersurfaces admitting non-isomorphic surjective endomorphisms. 
Inspired by \cite[Corollary 1.4]{Yan21}, we prove Theorem \ref{thm-singular-quadric}, which extends \cite[Proposition 8]{PS89} and has its  independent  interests. 
Let us begin with a historical remark on the (total) invariance of quadric hypersurfaces.
\begin{remark}\label{rem-quadric-totally}
Let $f:\mathbb{P}^n\to \mathbb{P}^n$ be a non-isomorphic surjective endomorphism.
\begin{itemize}
\item Paranjape and Srinivas (\cite[Proposition 8]{PS89}) showed  that there is no non-isomorphic surjective endomorphism of  smooth quadrics of dimension $\ge 3$.  
\item Nakayama and Zhang showed that there are no quadrics  (either smooth or not) in $\mathbb{P}^3$ which is $f^{-1}$-invariant; see \cite[Theorem 1.5 (5)]{NZ10}.
\item Yanis extended \cite[Theorem 1.5 (5)]{NZ10} and showed  that there is no  $f^{-1}$-invariant quadrics (no matter smooth or not) in $\mathbb{P}^n$ with $n\ge 4$; see \cite[Corollary 1.4]{Yan21}.
\end{itemize}
\end{remark}

Before proving Theorem \ref{thm-singular-quadric}, 
we refer the reader to Notation \ref{notation2.1} for the rank and the dimension of the singular locus of a given quadric.

\begin{proof}[Proof of Theorem \ref{thm-singular-quadric}]
For the one direction, if $k=3$ (resp. $k=2$), we may assume that  $X$ is defined by $x_0x_1=x_2x_3$ (resp. $x_1^2=x_0x_2$).  
Then the coordinate map 
$[x_0:\cdots:x_{n}]\mapsto [x_0^q:\cdots:x_n^q]$ 
induces a polarized endomorphism. 
If $k=1$, $X$ is the union of two hyperplanes and  also admits a polarized endomorphism. 
We may assume $k\ge 4$ in the following and show the converse direction. 
If $k=n$, then $X$ is smooth and hence our theorem follows from \cite[Proposition 8]{PS89}.
So we may assume $4\le k<n$ so that 
$X$ is a (singular) cone over a smooth  quadric of dimension $k-1$. 
Besides, the singular locus $L$ of $X$ is a linear subspace of codimension $(k+1)$ in $\mathbb{P}^n$ which is defined by 
$x_0=\cdots=x_k=0$.

Suppose to the contrary that $f:X\to X$ is a non-isomorphic surjective endomorphism.
Since $\dim (X)=n-1\ge 3$ by assumption, by the Lefschetz theorem for Picard group,   
we have $\rho(X)=1$.
Hence, $f$ is $q$-polarized with $q>1$.
Further, the ample generator $H_X$ of $\textup{Pic}(X)$ being very ample and the natural restriction $H^0(\mathbb{P}^n,\mathcal{O}_{\mathbb{P}^n}(m))\to H^0(X,mH_X)$ being surjective for any $m$ (cf.~\cite[Chapter III, Theorem 5.1]{Har77}), it follows from \cite[Proposition 2.1]{Fak03}  that $f$ induces a polarized endomorphism $g$ on $\mathbb{P}^n$ such that $f:=g|_{X}$. 

Let $H_1,\cdots,H_{n-k}$ be  general hyperplanes  on $\mathbb{P}^{n}$ such that the linear complete intersection $S:=H_1\cap\cdots\cap H_{n-k}$ is disjoint with both $L$ and $g^{-1}(L)$ by the dimension counting, noting that $\dim(L)=\dim(g^{-1}(L))=n-k-1$. 
Then $S$ is defined by the solutions of the following system of linear equations
\[
\begin{pmatrix}
b_{1,0}&\cdots&b_{1,n}\\
\vdots&\ddots&\vdots\\
b_{n-k,0}&\cdots&b_{n-k,n}
\end{pmatrix}
\begin{pmatrix}
x_{0}\\
\vdots\\
x_n
\end{pmatrix}=0.
\]
We note that, the matrix on the left hand side is of full rank since $\dim (S)=k$ and 
each hyperplane $H_i$ is defined by $\sum_{j=0}^nb_{i,j}x_j=0$. 
Since $S$ is disjoint with $L$, we see that
the system of linear equations below has only trivial solution.
\[
\begin{pmatrix}
b_{1,k+1}&\cdots&b_{1,n}\\
\vdots&\ddots&\vdots\\
b_{n-k,k+1}&\cdots&b_{n-k,n}
\end{pmatrix}
\begin{pmatrix}
x_{k+1}\\
\vdots\\
x_n
\end{pmatrix}=0
\]
In particular, the square matrix
\[
\begin{pmatrix}
b_{1,k+1}&\cdots&b_{1,n}\\
\vdots&\ddots&\vdots\\
b_{n-k,k+1}&\cdots&b_{n-k,n}
\end{pmatrix}
\]
is nonsingular, and hence the linear subspace  $S$ is actually defined by the solution of 
\[
\begin{pmatrix}
x_{k+1}\\
\vdots\\
x_n
\end{pmatrix}=\begin{pmatrix}
a_{1,0}&\cdots&a_{1,k}\\
\vdots&\ddots&\vdots\\
a_{n-k,0}&\cdots&a_{n-k,k}
\end{pmatrix}
\begin{pmatrix}
x_0\\
\vdots\\
x_k
\end{pmatrix}.
\]
Consider the following projection map (cf.~\cite[Proof of Corollary 1.4]{Yan21})
\begin{align*}
\tau:\mathbb{P}^n&\dashrightarrow S\cong\mathbb{P}^k\subseteq\mathbb{P}^{n}\\
[x_0:\cdots:x_n]&\mapsto[x_0:\cdots:x_{k}:\sum_{i=0}^{k}a_{1,i}x_i:\cdots:\sum_{i=0}^ka_{n-k,i}x_i]
\end{align*}
where the only non-defined locus of $\tau$ is $L$. 
Then the composite dominating map $g_S:=\tau\circ g|_S$ is  not well-defined at $s\in S$ if and only if $g(s)\in L$. 
Since $g^{-1}(L)$ is disjoint with $S$, it follows that $g_S$ is a (finite) surjective endomorphism on $S$.
We claim that  $\deg(g_S)>1$. 
Suppose to the contrary that  $\deg(g_S)=1$.  
Then $S\to g(S)$ is a finite birational morphism and hence $g|_S$ is the normalization due to the normality of $S$. 
However, there is a further finite birational map $g(S)\to S$ such that the composite $\tau\circ g|_S$ is an isomorphism. 
This in turn implies $S\cong g(S)$; in particular, the restriction $g|_S$ is an isomorphism which is absurd since $g|_S$ is also $q$-polarized with $q>1$ (cf.~\cite[Proposition 1.1]{MZ18}). 
Therefore, $g_S$ is a non-isomorphic surjective endomorphism and hence polarized since $S\cong \mathbb{P}^k$. 
Now that $g(X)=X$, from the defining equation of $X$, we have $\tau(X\cap g(S))\subseteq X$.
Hence,
$$g_S(X\cap S)=\tau(g(X\cap S))\subseteq \tau(X\cap g(S))\subseteq X\cap S.$$
By Bertini's theorem, $X\cap S$ is a reduced and irreducible subvariety (cf.~\cite[Chapter III, Remark 7.9.1]{Har77}). 
Moreover, since $X$ is smooth in codimension $k$, the intersection $X\cap S$ is smooth by the choice of $S$. 
To conclude, we get a non-isomorphic surjective endomorphism $g_S|_{X\cap S}$ on the smooth quadric $X\cap S$ of dimension $k-1\ge 3$ (cf.~\cite[Proposition 1.1]{MZ18}) which gives rise to a contradiction to \cite[Proposition 8]{PS89}.
\end{proof}

We end up this section with the following example of Theorem \ref{thm-singular-quadric} in a theoretical way.
\begin{example}
\textup{
Let $Y$ be a smooth quadric surface and $W=\mathbb{P}_Y(\mathcal{O}\oplus \mathcal{O}(-1,-1))$ a $\mathbb{P}^1$-bundle over $Y$.
It is known that $W$ is a Fano threefold.
Let $S$ be the negative section of $W$ such that $S|_S=(-1,-1)$. 
Then there is a blow-down $\tau: W\to X$  contracting $S$ to a single point, i.e., the vertex of the singular cone $X$ over $Y$ (cf.~\cite[Example 2.7]{KM98}). 
Clearly, $W$ admits a polarized endomorphism  and it descends to  $X$ after iteration, noting that $\tau$ is the composite of a $K$-negative divisorial contraction and a $K$-negative small contraction  (cf.~\cite[Lemmas 6.4 and 6.5]{MZ18}).
Consequently, the $3$-dimensional singular quadric cone $X$  admits a non-isomorphic surjective endomorphism.}
\end{example}


\section{Higher Picard number case, Proof of Theorem \ref{thm-middle-index}}\label{sec-high-picard}

In this section, we  study del Pezzo manifolds and Mukai manifolds of Picard number $\ge 2$ so as to prove Theorem \ref{thm-middle-index}. 
Our first result of this section is Theorem \ref{thm-higher-pic}, which answers  Question \ref{main-ques-toric} positively when $X$ is of Picard number $\ge 2$ and of index $\ge \dim(X)-2$. 

\begin{theorem}\label{thm-higher-pic}
Let $X$ be a  Fano manifold of Picard number $\ge 2$.
Suppose that the index $i(X)\ge\dim(X)-2$.
If $X$ admits an int-amplified endomorphism, then $X$ is toric.
\end{theorem}

Theorem \ref{thm-higher-pic} follows from Propositions \ref{prop-del-pezzo-higher-rho} and \ref{prop-mukai-higher-rho} below.

\begin{proposition}\label{prop-del-pezzo-higher-rho}
Theorem \ref{thm-higher-pic} holds if $X$ is a del Pezzo manifold.
\end{proposition}

\begin{proof}
If $\dim(X)\le 3$, then our proposition follows from \cite[Theorem 3]{Nak02} and \cite[Theorem 1.4]{MZZ22}. 
So we may assume that $\dim(X)\ge 4$. 
Then by Theorem \ref{class-del-pezzo}, $X$ has to be $\mathbb{P}^2\times\mathbb{P}^2$, which is clearly toric.
\end{proof}

\begin{proposition}\label{prop-mukai-higher-rho}
Theorem \ref{thm-higher-pic} holds if $X$ is a Mukai manifold. 
\end{proposition}

Before  proving Proposition \ref{prop-mukai-higher-rho}, we  show the following Lemmas \ref{lem-double-cover} and \ref{lem-middle-index-bidegree11}.

\begin{lemma}[{cf.~\cite[Proof of Lemma 4.12]{MZZ22}}]\label{lem-double-cover}
Let $X$ be a double cover of $\mathbb{P}^n\times\mathbb{P}^n$ with $n\ge 2$ whose branch locus is a divisor of bidegree $(2,2)$.
Then $X$ does not admit any int-amplified endomorphism.
\end{lemma}

\begin{proof}
Suppose to the contrary that $X$ admits an int-amplified endomoprhism $f$.
By the ramification divisor formula,  such $X$ is a  Fano manifold of dimension $2n$ and of  index $n$. 
Let $Y:=Y_1\times Y_2=\mathbb{P}^n\times\mathbb{P}^n$ and $\pi:X\to Y$ the double cover.
Let $p_i:Y\to Y_i\cong\mathbb{P}^n$ be the two projections and  $\pi_i:X\to Y_i$ the two Fano contractions.
Then  $f$ descends to an int-amplified endomorphism $g_i$ on $Y_i$ after iteration (cf.~\cite[Theorem 1.10]{Men20}).
Let $g:=g_1\times g_2$ be the induced map on $Y$.
We have the following commutative diagram.
\[
\begin{tikzpicture}[scale=0.8]
    \node (ax) at (4,4.5) {\rotatebox[origin=c]{180}{\(\circlearrowleft\)}};
    \node (lx) at (4,5) {\(f\)};
    \node (ay) at (4,1.5) {\(\circlearrowright\)};
    \node (ly) at (4,1) {\(g\)};
    \node (X) at (4,4) {\(X\)};
    \node (Y) at (4,2) {\(Y=Y_1 \times Y_2\)};
    \node (Y1) at (0,0) {\(g_1\leftacts Y_1\)};
    \node (Y2) at (8,0) {\(Y_2\rightacts g_2\)};

    \path[-to,font=\scriptsize]
    (X) edge node[right] {\(\pi\)} (Y)
    edge node[pos=0.55,yshift=10pt] {\(\pi_1\)} (Y1)
    edge node[pos=0.55,yshift=10pt] {\(\pi_2\)} (Y2);
    \draw [-to,font=\scriptsize] (Y.south west) --  node[pos=0.4,yshift=-10pt] {\(p_1\)} (Y1.north east);
    \draw [-to,font=\scriptsize] (Y.south east) --  node[pos=0.4,yshift=-10pt] {\(p_2\)} (Y2.north west);
\end{tikzpicture}
\]
Let $B=B_\pi$ be the branch divisor of $\pi$, which is non-empty since $Y$ is simply connected. 
\begin{claim}\label{claim-lem-double}
Write $B=\cup B_i\subseteq Y$ as a union of  prime divisors.
Then each component $B_i$ dominates both $Y_1$ and $Y_2$.
Further, each $B_i$ does not lie in the branch locus $B_g$ of $g$.
\end{claim}
\begin{proof}
Suppose $B_i$ comes from $Y_1$ (similar arguments work for $Y_2$), i.e., $B_i=p_1^*(p_1(B_i))$. 
By the cone theorem (cf.~\cite[Theorem 3.7]{KM98}), $\pi^*B_i=\pi_1^*(p_1(B_i))$ is reduced and irreducible (cf.~\cite[Proposition 2.17]{MZZ22}); in particular, $B_i$ does not lie in the branch locus of $\pi$, a contradiction.
Hence, each $B_i$ dominates both $Y_1$ and $Y_2$.  
The first half of our claim is proved.
For the second half, note that  the branch locus $B_g$ of $g$ is  $p_1^*(B_{g_1})+p_2^*(B_{g_2})$, where $B_{g_1}$ and $B_{g_2}$ are the branch divisors of $g_1$ and $g_2$, respectively. 
By the first half of our claim, we have $B_i\not\subseteq B_g$ and the second half is thus proved.
\end{proof}
Now we come back to proving Lemma \ref{lem-double-cover}. 
By Claim \ref{claim-lem-double}, each irreducible component $B_i$ is not contained in $B_g$ and therefore we have the following reduced decomposition
$$g^*B_i=g^{-1}(B_i)=\sum_k B_i(k) $$
with $B_i(k)$ being distinct prime divisors on $Y$. 
Since $B_i\subseteq B_\pi\subseteq B_{\pi\circ f}=B_{g\circ\pi}$  but $B_i\not\subseteq B_g$ for each $i$, there exists at least one component in $g^*B_i$ which lies in the branch locus of $\pi$.
Without loss of generality, one may assume that $B_i(1)\subseteq B_\pi$ for each $i$.
Since $B_i(1)\neq B_j(1)$ for any $i\neq j$, we have $B_\pi=\cup_iB_i(1)$ by the bijection of the components; hence, $g(B_\pi)=B_\pi$. 
Since $\deg(\pi)=2$, we further have $\pi^*B_\pi=2\pi^{-1}(B_\pi)$.
Then 
$$2f^*\pi^{-1}(B_\pi)=f^*\pi^*B_\pi=\pi^*g^*B_\pi=\pi^*(g^{-1}(B_\pi)).$$
Comparing the coefficients of both sides, one has $g^{-1}(B_\pi)\subseteq B_\pi$.
In particular,  we have $g^{-1}(B_\pi)=B_\pi$.
However, since $g$ is int-amplified,  every component of $B_\pi$ is contained in the branch locus of $B_g$ (cf.~\cite[Theorem 1.1]{Men20}), a contradiction to Claim \ref{claim-lem-double}.
\end{proof}

\begin{lemma}\label{lem-middle-index-bidegree11}
Let $X$ be a  Fano manifold of dimension $2r$ and of  index $r$.
If $X$ is the intersection of two smooth divisors in $\mathbb{P}^{r+1}\times\mathbb{P}^{r+1}$ of bidegree $(1,1)$, then  $X$  does not admit any int-amplified endomorphism.
\end{lemma}

\begin{proof}
It is known that $\rho(X)=2$  and  $X$ admits two non-equidimensional Fano contractions onto $\mathbb{P}^{r+1}$ (cf.~\cite[p.~24]{Wis94} and \cite[p.~151]{IP99}). 
Suppose that $X$ admits an int-amplified endomorphism $f$. 
Denote by $\pi_i$ the two Fano contraction onto $Y_i\cong\mathbb{P}^{r+1}$.
By \cite[Theorem 6.8]{BW96} and \cite[Theorem 5.2, Subsection 9.1.2 Case (3)]{Lan98}, there are $r+2$ isolated $r$-dimensional special fibres $\mathbb{P}^r$ for each $\pi_i$ (and other fibres are all $\mathbb{P}^{r-1}$). 
We denote by $F_1,\cdots,F_{r+2}$ the special fibres of $\pi_1:X\to Y_1$.
By the equivariant minimal model program (\cite[Theorem 1.10]{Men20}; cf.~\cite[Theorem 1.8]{MZ18}), after iteration,  all of $F_i$ are $f^{-1}$-invariant. 
By \cite[Lemma 7.5]{CMZ20}, after iteration, 
$E_i:=\pi_2(F_i)\subseteq Y_2$ is a $(g_2:=f|_{Y_2})^{-1}$-invariant divisor, noting that $F_i$ is not $\pi_2$-contracted.

We claim that $E_i$ and $E_j$ are mutually different.
Suppose the claim for the time being.
Then by \cite[Corollary 1.2]{Zho21}, we see that $(Y_2,\sum_{i=1}^{r+2} E_i)$ is a toric pair with each $E_i\cong \mathbb{P}^r$ and $\sum E_i$ is  a simple normal crossing divisor.
Let $v\in Y_2$ be a special point such that $X_v:=\pi_2^{-1}(v)\cong\mathbb{P}^r$ is a singular fibre. 
By \cite[Theorem 1.10]{Men20}, $g_2^{-1}(v)=v$ after iteration and hence $v$ lies in the branch locus of $g_2$. 
On the other hand, since $g_2: Y_2\backslash\sum E_i\to Y_2\backslash \sum E_i$ is \'etale (cf.~\cite[Theorem 1.4]{MZg23}), we have $v\in \sum E_i$.
We may assume that $v\in S:=E_{i_1}\cap\cdots\cap E_{i_k}$ but $v\not\in E_j$ for any $j\not\in\mathcal{S}:=\{i_1,\cdots,i_k\}$. 
As $S$ is also $g_2^{-1}$-invariant and $(S,\sum_{j\not\in\mathcal{S}}E_j|_S)$ is also a toric pair,  the restriction $g_2|_S:S\backslash\sum_{j\not\in\mathcal{S}}E_j|_S\to S\backslash\sum_{j\not\in\mathcal{S}}E_j|_S$ is \'etale (cf.~[Ibid.]).   
In particular, $k=r+1$ and 
$$v=E_1\cap\cdots\cap E_{j-1}\cap E_{j+1}\cap\cdots\cap E_{r+2}$$
for some $j$. 
Then we get a contradiction by the following equality
$$\mathbb{P}^r\cong X_v=\pi_2^{-1}(v)=\pi_2^*E_1\cap\cdots\cap \pi_2^*E_{j-1}\cap\pi_2^*E_{j+1}\cap\cdots\cap \pi_2^*E_{r+2},$$
noting that the right hand side is the intersection of $r+1$ different Cartier divisor and thus is of dimension $2r-(r+1)=r-1$. 

So we are remained to show $E_i~(i=1,..,r+2)$ are mutually distinct. 
Note that $X$ can also be regarded as a bidegree $(1,1)$ divisor on $W:=\mathbb{P}(T_{\mathbb{P}^{r+1}})\subseteq\mathbb{P}^{r+1}\times\mathbb{P}^{r+1}$. 
Then $F_i$ is also the fiber of the  projection $\varphi_1 :W\to Y_1$. 
The other contraction $\varphi_2:W\to Y_2$
is given by $(y,\xi)\mapsto H$, where $y\in Y_1$, $\xi$ is a hyperplane of $T_{Y_1,y}$ and $H$ is the unique hyperplane of $Y_1$ passing through  $y$ such that $T_{H,y}=\xi$. 
Here, we identify $Y_2$ with the dual space of $Y_1$. 
So the $\varphi_2$-images of distinct fibers of $\varphi_1$  are distinct hyperplanes. 
In particular, $E_i$ are mutually distinct. 
We finish the proof of our claim and also the lemma. 
\end{proof}

\begin{proof}[Proof of Proposition \ref{prop-mukai-higher-rho}]
Let $f:X\to X$ be an int-amplified endomorphism. 
In  the view of \cite[Theorem 3]{Nak02} and \cite[Theorem 1.4]{MZZ22}, we may assume $\dim(X)\ge 4$. 
We shall use the classification in \cite[Section 3]{Muk89}. 
Since $\rho(X)\ge 2$, by \cite[Theorem 7]{Muk89}, $X$ is either of dimension 4 and of product type, or $X$ is isomorphic to a linear section of one of the following  (where $Q^n$ denotes a smooth quadric  of dimension $n$). 
\begin{enumerate}
\item a double cover of $\mathbb{P}^2\times\mathbb{P}^2$ whose branch locus is a divisor of bidegree $(2,2)$;
\item a divisor of $\mathbb{P}^2\times\mathbb{P}^3$ of bidegree $(1,2)$;
\item $\mathbb{P}^3\times\mathbb{P}^3$;
\item $\mathbb{P}^2\times Q^3$;
\item the blow-up of $Q^4\subseteq \mathbb{P}^5$ along a conic not contained in a plane of $Q^4$;
\item the flag variety of $\textup{Sp}(2)/B$, where $B$ is a Borel subgroup. 
\item the blow-up of $\mathbb{P}^5$ along a line.
\item the $\mathbb{P}^1$-bundle $\mathbb{P}_{Q^3}(\mathcal{O}\oplus\mathcal{O}(1))$ over $Q^3\subseteq\mathbb{P}^4$;
\item the $\mathbb{P}^1$-bundle $\mathbb{P}_{\mathbb{P}^3}(\mathcal{O}\oplus\mathcal{O}(2))$ over $\mathbb{P}^3$.
\end{enumerate}

If $\dim(X)=4$ and is of product type $X_1\times X_2$, then by \cite[Theorem 1.10]{Men20}, $f$ descends to both  $X_i$ after iteration and hence both $X_i$ 
are toric (cf.~\cite[Theorem 1.4]{MZZ22} and \cite[Theorem 3]{Nak02}), noting that $\dim(X_i)\le 3$.
Therefore, we may assume in the following that $X$  is one of Case $(1)\sim(9)$  listed above or their linear sections.

If $X$ is of Case (1), then our theorem follows from Lemma \ref{lem-double-cover}. 
If $X$ is of Case (5) or (8), then $\dim(X)=4$ and by \cite[Theorem 1.10]{Men20}, $f$ descends to $Q^3$ or $Q^4$, a contradiction to \cite[Proposition 8]{PS89}. 
If $X$ is of Case (9), then $\dim(X)=4$ and $X$ is a splitting $\mathbb{P}^1$-bundle over $\mathbb{P}^3$ and hence toric (cf.~\cite[Proposition 2.9]{MZZ22}).
So we verify Case (9) and exclude Case (1), (5) and (8).

Suppose that $X$ is of Case (2), i.e., $X$ is a divisor on $\mathbb{P}^2\times\mathbb{P}^3$ of bidegree $(1,2)$.
Then $\dim(X)=4$ and $X$ admits two Fano contractions $\pi$ onto $Z:=\mathbb{P}^3$ and $\tau$ onto $W:=\mathbb{P}^2$.
By \cite[Theorem 6.8]{BW96} and \cite[Theorem 5.2, Subsection 9.1.2 Case (4)]{Lan98}, there are 8 singular $2$-dimensional fibres $F_1,\cdots, F_8$ of $\pi$, all of which dominate $W$. 
By \cite[Theorem 5.1]{Fak03} and \cite[Theorem 1.10]{Men20}, we can pick a general $f|_W$-periodic point $w\in W$ and after iteration, its fibre $X_w$ admits an int-amplified endomorphism $f_w:=f|_{X_w}$. 
Note that $X_w\cong\mathbb{P}^1\times\mathbb{P}^1$ is a smooth quadric surface (cf.~\cite[Subsection 9.1.2 Case (4)]{Lan98}). 
On the other hand, since every $F_i$ dominates $W$ along the second contraction, it follows that $F_i\cap X_w\neq \emptyset$ and hence $F_i\cap X_w$ contains either one  $f_w^{-1}$-invariant point or one $f_w^{-1}$-invariant line after iteration. 
By \cite[Theorem 3.2]{MZZ22}, there are at most 4 $f_w^{-1}$-invariant lines.
Since $X_w\cong \ell_1\times\ell_2\cong \mathbb{P}^1\times\mathbb{P}^1$, at least one of $\ell_i$ has more than $3$ $(f|_{\ell})^{-1}$-invariant points after iteration (cf.~\cite[Theorem 1.10]{Men20}).
This leads to a contradiction (cf.~e.g.~\cite[Theorem 1.1]{Zho21}). 
So Case (2) is excluded.

Suppose that $X$ is of Case (3). 
If $\dim (X)=6$, then $X\cong\mathbb{P}^3\times\mathbb{P}^3$ and hence toric.  
If $\dim (X)=5$, then $X\cong\mathbb{P}(T_{\mathbb{P}^3})$, which is absurd since the tangent bundle $T_{\mathbb{P}^3}$ does not split (cf.~\cite[Proposition 3]{Ame03}). 
If $\dim (X)=4$, then $X$ is a complete intersection of two divisors of bidegree $(1,1)$ on $\mathbb{P}^3\times\mathbb{P}^3$ and by Lemma \ref{lem-middle-index-bidegree11}, we get a contradiction.

Suppose that $X$ is of Case (4).
If $\dim(X)=5$, i.e., $X\cong \mathbb{P}^2\times Q^3$, then by \cite[Theorem 1.10]{Men20}, we get a contradiction by noting that $Q^3$ does not admit a non-isomorphic surjective endomorphism (cf.~\cite[Proposition 8]{PS89}). 
So we may assume $\dim(X)=4$ and then 
$X$ is a linear section of $\mathbb{P}^2\times Q^3$, i.e., a divisor of bidegree $(1,1)$.
Then $X$ admits a Fano contraction to $Q^3$; see \cite[p. 151]{IP99}.
In particular, our assumption contradicts \cite[Proposition 8]{PS89}.
So we exclude Case (4).

Suppose that $X$ is of Case (6), i.e., the flag variety of $\textup{Sp}(2)$.
Since $\textup{Sp}(2)$ has two simple roots and $\rho(X)\ge 2$ by assumption, if $P$ is a parabolic subgroup such that $X\cong\textup{Sp}(2)/P$, then $P$ is determined by the two simple roots, in which case $P$ is a Borel subgroup of $\textup{Sp}(2)$.
Then $\dim (X)$ is the number of positive roots, which is 4. 
So $X$ is a rational homogeneous fourfold and our assumption is absurd  (cf.~\cite[Proposition 2]{PS89}).

Finally, suppose that $X$ is of type (7).
If $\dim (X)=5$, then it is a blow-up of $\mathbb{P}^5$ along a line (lying in a toric boundary) and hence $X$ is toric. 
If $\dim (X)=4$, then it is a linear section of the above variety and by  adjunction,  $X$ is a blow-up of $Q^4\subseteq\mathbb{P}^5$ along a line; in particular, $f$ descends to $Q^4$ after iteration, which is absurd (cf.~\cite[Proposition 8]{PS89}). 
So we finish the proof of our proposition.
\end{proof}

\begin{proof}[Proof of Theorem \ref{thm-middle-index}]
By Lemma \ref{lem-index-reduction}, we may assume that $n=2r-1$ or $n=2r$.
If $n=2r-1$, by Lemma \ref{lem-index-reduction} again,   $X$ is toric if and only if $X\cong\mathbb{P}_{\mathbb{P}^r}(\mathcal{O}(2)\oplus\mathcal{O}(1)^{\oplus(r-1)})$ if and only if $X$ admits an int-amplified endomorphism (cf.~\cite[Proposition 3]{Ame03} and \cite[Proposition 8]{PS89}). 
In the following, we may assume that $n=2r$.
We may further assume that $n\ge 6$ in the view of Proposition \ref{prop-mukai-higher-rho}, since Mukai fourfolds have  index 2. 
We shall apply the classification in Lemma \ref{lem-class-middleindex}.

First, we assume $X$ is of type (1).
If $X$ is  $\mathbb{P}_{\mathbb{P}^{r+1}}(\mathcal{O}(2)^{\oplus 2}\oplus\mathcal{O}(1)^{\oplus(r-2)})$ or $\mathbb{P}_{\mathbb{P}^{r+1}}(\mathcal{O}(3)\oplus\mathcal{O}(1)^{\oplus(r-1)})$, then  it is toric and admits an int-amplified endomorphism.
If $X$ is  $\mathbb{P}_{Q^{r+1}}(\mathcal{O}(2)\oplus\mathcal{O}(1)^{\oplus (r-1)})$ or $\mathbb{P}_{Q^4}(\mathbb{E}(1)\oplus\mathcal{O}(1))$, then it is not toric since the smooth quadric $Q$ (of dimension $\ge 3$) is not toric, noting that the big torus action descends along a fibration;  
further, it does not admit an int-amplified endomorphism since so does $Q$ (cf.~\cite[Theorem 1.10]{Men20} and \cite[Proposition 8]{PS89}). 
If $X$ is $\mathbb{P}_{V_d}(\mathcal{O}(1)^{\oplus r})$, then it is not toric and does not admit any int-amplified endomorphism (cf.~Theorem  \ref{main-del-pezzo-non}), noting that the only smooth projective toric variety of Picard number 1 is the projective space.

Second, assume $X$ is of type (2). 
Suppose that $X$ admits an int-amplified endomorphism $f$. 
Then $f$ descends to $Y$ (cf.~\cite[Theorem 1.10]{Men20}) and it follows from \cite[Theorem 5.1]{Fak03} that one can pick a general $f$-periodic fibre $F$ such that $f|_F$ is still int-amplified after iteration (cf.~\cite[Lemma 2.3]{Men20}). 
However, in this case, $F$ is a smooth quadric in $\mathbb{P}^{r+1}$ (with $r\ge 3$) and thus does not admit any non-isomorphic surjective 
endomorphism (cf.~\cite[Proposition 8]{PS89}).   
So our assumption is absurd; in particular, such $X$ is not toric (cf.~e.g.~\cite[Theorem 1.4 and its proof]{MZg23}).

Finally, assume that $X$ is of type (3). 
If $X$ is in the former case of type (3), then $X$ does not admit any non-isomorphic surjective endomorphism $f$; for otherwise, such $f$ would descend to a smooth quadric of dimension $\ge 6$ (cf.~\cite[Theorem 1.10]{Men20}), contradicting  \cite[Proposition 8]{PS89}.
If $X$ is in the latter case of type (3), our theorem follows from 
Lemma \ref{lem-middle-index-bidegree11}.
As a result,  $X$ is not toric, either. 
\end{proof}

We end up this section with the following remark.
\begin{remark}[Connections with \cite{JZ23}]
In \cite[Section 1]{JZ23}, Jia and the second author expected that if $X$ is a  Fano fourfold equipped with a $\mathbb{P}^1$-fibration $\pi$, then the dynamical condition on $X$ will make $\pi$ be a conic bundle, i.e., every fibre is  a conic in $\mathbb{P}^2$.
As indicated by Lemma \ref{lem-middle-index-bidegree11}, our expectation is reasonable. 
Besides, if $\pi$ is a singular conic bundle,  the existence of the rational curve $\ell$ with $-K_X\cdot \ell=1$ (cf.~\cite[Theorem 5.1]{JZ23}) forces the index $i(X)=1$.
So our Theorem \ref{thm-middle-index}  complements \cite[Theorem 1.4]{JZ23}.
\end{remark}

\section{Endomorphisms of Fano manifolds of small Fano index}\label{sec-index<=2}
In the last section, we study Conjecture \ref{main-conj-pn} when the Fano manifolds have small index. 
Propositions \ref{thm-akp-index<=2} and \ref{prop-nor-bdle-negative} are our main results of this section. 
The technical assumption imposed here is inspired by \cite[Proposition 2.1]{Ame97}.

\begin{proposition}\label{thm-akp-index<=2}
Let $X$ be a Fano manifold which is of Picard number 1, of dimension $n$ and of  index $i(X)=2$ (resp.\,$i(X)=1$).  
Let $f:X\to X$ be a surjective endomorphism.
If there is a line $\ell$ (resp.\,a conic) such that the inverse image $f^{-1}(\ell)$ is not contained in the ramification locus of $f$, then $\deg(f)=1$.
\end{proposition}

\begin{proof}
Suppose  that $X$ admits a non-isomorphic  surjective endomorphism $f$ such that $f^{-1}(\ell)$ is not contained in the ramification locus. 
Then the inverse image $\ell':=f^*(\ell)$ contains a reduced and irreducible component $\ell_0$. 
We analyze the  normal bundle sequence, 
$$0\to\mathcal{O}_\ell(2)\to T_X|_\ell\to \mathcal{N}_{\ell/X}\to 0.$$
Suppose first that the index $i(X)=2$. 
Then we have $\det \mathcal{N}_{\ell/X}=0$.
If $\mathcal{N}_{\ell/X}=\mathcal{O}_\ell^{\oplus(n-1)}$, then such $X$ contains  a rational curve with trivial normal bundle; in particular, $\deg(f)=1$ (cf.~\cite[Theorem 2]{HM03}).
Therefore, we may assume $\mathcal{N}_{\ell/X}=\oplus\mathcal{O}_\ell(a_i)$ with $a_1\ge \cdots \ge a_{n-1}$ and $a_{n-1}<0$ (which always holds if either $i(X)=1$; see Lemma \ref{lem-normalbdle-index1},  or $i(X)=2$ and $X$ does not contain any rational curve with trivial normal bundle). 

Choosing $\ell'$ and $\ell_0$ as above, we have a natural morphism
$$\varphi_1:\mathcal{N}_{\ell_0/X}\to \mathcal{N}_{\ell'/X}|_{\ell_0}=f^*\mathcal{N}_{\ell/X}|_{\ell_0}=\oplus\mathcal{O}_{\ell_0}(qa_i).$$
Note  that $\varphi_1$ is an isomorphism around a smooth point of $\ell_0$.
Since $\ell_0$ is reduced and hence a local complete intersection, the following map
$$\varphi_2: T_X|_{\ell_0}\to\mathcal{N}_{\ell_0/X}$$
is surjective around the generic point of $\ell_0$.
Fix an embedding $i:X\hookrightarrow \mathbb{P}^N$ with respect to the very ample divisor $-uK_X$ for some positive integer $u$.
Then $\Omega_X(-2uK_X)$ is globally generated.
As $T_X(K_X)=\wedge^{n-1}\Omega_X$, we have $T_X((-2u(n-1)+1)K_X)$ is globally generated. 
In particular, $T_X(m)$ is globally generated for some positive integer $m$ and then it follows from the generic surjectivity of $\varphi_2$ that 
$\mathcal{N}_{\ell_0/X}(m)$ is generically globally generated; hence, $q\le -\frac{m}{a_{n-1}}$. 
However, this is impossible since our $q$ can be arbitrarily large.
\end{proof}

With a few modifications, we end up the paper by extending Proposition \ref{thm-akp-index<=2} to the following. 
We expect that Proposition \ref{prop-nor-bdle-negative} could play a  role for us to study Fano manifolds of small  index, although in general the covering locus of the  lines of the negative type  is quite complicated to describe. 
\begin{proposition}\label{prop-nor-bdle-negative}
Let $X$ be a Fano manifold of Picard number 1, of dimension $n$ and of  index $i(X)$. 
Let $H$ be the fundamental divisor of $X$, i.e., $-K_X\sim i(X)H$. 
Suppose that $X$ satisfies the following conditions.
\begin{enumerate}
\item There is a family $\mathcal{C}$ of lines such that the covering locus is an effective divisor $D$.
\item  $D\sim \lambda H$ with $\lambda\ge 2\cdot i(X)$.
\item The normal bundle of a general member $[\ell]\in \mathcal{C}$ in $X$ has a negative summand.
\end{enumerate}
Then $X$ does not admit any non-isomorphic surjective endomorphism.
\end{proposition}

\begin{proof}
Suppose to the contrary that $f:X\to X$ is a non-isomorphic surjective  endomorphism which is $q$-polarized with $q>1$. 
We show that $f^{-1}(D)$ is not contained in the ramification  locus of $f$. 
Indeed, by \cite[Theorem 1.10]{MZ18}, 
$R_f\sim (q-1)i(X)H$. 
Hence, if $f^{-1}(D)$ is contained in $R_f$, then $f^*D=\sum_i a_iD_i$ with every $a_i\ge 2$.
Then 
$$R_f\ge \sum_i (a_i-1)D_i\ge \sum_i\frac{1}{2}a_iD_i=\frac{1}{2}f^*D.$$
Then we have (in the numerical sense of effective $(n-1)$-cycles)
$$(q-1)i(X)H\ge \frac{1}{2}f^*D\ge qi(X)H$$
 by our assumption, which gives rise to a contradiction.
Therefore, there is some line $\ell$ on $X$ such that the inverse image $\ell':=f^*(\ell)$ contains a reduced and irreducible component $\ell_0$.
The remaining proof is the same as in the proof of Proposition \ref{thm-akp-index<=2}.
\end{proof}


\begin{thebibliography}{99}  

\bibitem[Ame97]{Ame97}
E. Amerik, Maps onto certain Fano threefolds, {\em Doc. Math.}, \textbf{2}  (1997), 195--211.


\bibitem[Ame03]{Ame03}
E. Amerik, On endomorphisms of projective bundles,
{\em Manuscripta Math.}, \textbf{111} (2003), no. 1, 17--28.


\bibitem[Ame07]{Ame07}
E. Y. Amerik, Mappings onto quadrics, (Russian) {\em Mat. Zametki}, \textbf{81} (2007), no. 4, 621--624; translation in {\em Math. Notes} \textbf{81} (2007), no. 3-4, 549--552.

\bibitem[ARV99]{ARV99}
E. Amerik, M. Rovinsky and A. Van de Ven, A boundedness theorem for morphisms between threefolds, {\em Ann. Inst. Fourier (Grenoble)}, \textbf{49} (2) (1999), 405--415.


\bibitem[BW96]{BW96}
E. Ballico and J. Wi\'{s}niewski, On {B}\v{a}nic\v{a} sheaves and {F}ano manifolds, {\em Compos. Math.}, \textbf{102} (1996), no. 3, 313--335.



\bibitem[Bea01]{Bea01}
A. Beauville,
Endomorphisms of hypersurfaces and other manifolds,  
{\em Int. Math. Res. Not.}, 2001, no. 1, 53--58. 



\bibitem[BR86]{BR86}
M. Beltrametti and L. Robbiano,
Introduction to the theory of weighted projective spaces. 
{\em Exposition. Math.}, \textbf{4} (1986), no. 2, 111--162. 


\bibitem[Bra92]{Bra92}
R. Braun, On the normal bundle of {C}artier divisors on projective varieties, {\em Arch. Math. (Basel)}, \textbf{59} (1992), no. 4, 403--411.







\bibitem[CMZ20]{CMZ20}
P. Cascini, S. Meng and D.-Q. Zhang,
Polarized endomorphisms of normal projective threefolds in arbitrary characteristic,
{\em Math. Ann.}, \textbf{378} (2020), no. 1-2, 637--665.






\bibitem[Chu23]{Chu23}
K. Chung, Double lines in the quintic del Pezzo fourfold, {\em Bull. Korean Math. Soc.}, \textbf{60} (2023), no. 2, 485--494.

\bibitem[DH17]{DH17}
T. Dedieu and A. H\"{o}ring, Numerical characterisation of quadrics, {\em Algebr. Geom.}, \textbf{4} (2017), no. 1, 120--135.





\bibitem[EH16]{EH16}
D. Eisenbud and J. Harris, {\em  3264 and all that---a second course in algebraic geometry}, Cambridge University Press, Cambridge, 2016.




\bibitem[Fak03]{Fak03}
N.~Fakhruddin,
Questions on self-maps of algebraic varieties,
{\em J. Ramanujan Math. Soc.}, \textbf{18} (2003), no. 2, 109--122.




\bibitem[FM19]{FM19}
B. Fu and P. Monteroc, Equivariant compactifications of vector groups with high index, {\em C. R. Math. Acad. Sci. Paris}, \textbf{357} (2019), no. 5, 455--461.





\bibitem[Har77]{Har77}
R. Hartshorne, {\em Algebraic geometry}, 
Graduate Texts in Mathematics, No. 52. Springer-Verlag, New York-Heidelberg, 1977.






\bibitem[HM03]{HM03}
 J.-M. Hwang and N. Mok, Finite morphisms onto Fano manifolds of Picard number $1$ which have rational curves with trivial normal bundles, {\em J. Algebraic Geom.},  \textbf{12} (2003), 627--651.

\bibitem[HM04]{HM04}
J.-M. Hwang and N. Mok,
Birationality of the tangent map for minimal rational curves, 
{\em Asian J. Math.}, \textbf{8} (2004), no. 1, 51--63. 

\bibitem[HN11]{HN11}
J. M. Hwang and N. Nakayama, 
On endomorphisms of Fano manifolds of Picard number one,   
{\em Pure Appl. Math. Q.}, \textbf{7} (2011), no. 4, Special Issue: In memory of Eckart Viehweg, 1407--1426. 


\bibitem[IP99]{IP99}
V. A. Iskovskikh and Yu. G. Prokhorov, {\em Algebraic Geometry V, Fano varieties}, Encyclopaedia Math. Sci. \textbf{47}, Springer, Berlin, 1999.


\bibitem[JZ23]{JZ23}
J. Jia, G. Zhong, Amplified endomorphisms of Fano fourfolds, {\em Manuscripta Math.}, \textbf{172} (2023), 567--598.



\bibitem[Keb02]{Keb02}
S. Kebekus,
Families of singular rational curves, 
{\em J. Algebraic Geom.}, \textbf{11} (2002), no. 2, 245--256. 

\bibitem[KO73]{KO73}
S. Kobayashi and T. Ochiai, Characterizations of complex projective spaces and hyperquadrics, {\em J. Math. Kyoto Univ.}, \textbf{13} (1973), 31--47.



\bibitem[KO75]{KO75}
S. Kobayashi and T. Ochiai, Meromorphic mappings onto compact complex spaces of general type, {\em Invent. Math.}, \textbf{31} (1975), no. 1, 7--16.




\bibitem[Kol96]{Kol96}
J. Koll{\'a}r, {\em Rational curves on algebraic varieties}, Ergebnisse der Mathematik und ihrer Grenzgebiete. 3. Folge. A Series of Modern Surveys in Mathematics \textbf{32},  Springer-Verlag, Berlin, 1996.



\bibitem[KM98]{KM98}
J. Koll{\'a}r and S. Mori,
  {\em Birational geometry of algebraic varieties}, Cambridge Tracts in Math. \textbf{134},   Cambridge University Press, Cambridge, 1998.
  
\bibitem[KT23]{KT23}
T. Kawakami and B. Totaro, {\em Endomorphisms of varieties and Bott vanishing}, Preprint, \href{https://arxiv.org/abs/2302.11921}
{arXiv:2302.11921}.


\bibitem[Lan98]{Lan98}
A. Langer, Fano {$4$}-folds with scroll structure, {\em Nagoya Math. J.}, \textbf{150} (1998), 135--176.


\bibitem[Men20]{Men20}
S. Meng, 
Building blocks of amplified endomorphisms of normal projective 
  varieties, {\em Math. Z.}, \textbf{294}  (2020), no. 3-4, 1727--1747. 


\bibitem[MZ18]{MZ18}
S. Meng and D.-Q. Zhang, Building blocks of polarized endomorphisms of normal projective varieties, {\em Adv. Math.}, \textbf{325} (2018), 243--273.


\bibitem[MZZ22]{MZZ22}
S. Meng, D.-Q. Zhang, G. Zhong, Non-isomorphic endomorphisms of Fano threefolds, {\em Math. Ann.}, \textbf{383} (2022), no. 3-4, 1567--1596.


\bibitem[MZg23]{MZg23}
S. Meng and G. Zhong,
Rigidity of rationally connected smooth projective varieties from dynamical viewpoints, {\em Math. Res. Lett.}, \textbf{30} (2023), no. 2, 589--610.




\bibitem[Miy04]{Miy04} 
Y. Miyaoka, Numerical characterisations of hyperquadrics, Complex analysis in several variables—Memorial Conference of Kiyoshi Oka's Centennial Birthday, 209–235, {\em Adv. Stud. Pure Math.} \textbf{42}, Math. Soc. Japan, Tokyo, 2004.



  
\bibitem[Muk89]{Muk89}
S. Mukai, Biregular classification of {F}ano {$3$}-folds and {F}ano manifolds of coindex {$3$}, {\em Proc. Nat. Acad. Sci. U.S.A.}, \textbf{86} (1989), no. 9, 3000--3002.


\bibitem[Nak02]{Nak02}
N.~Nakayama,
Ruled surfaces with non-trivial surjective endomorphisms,
{\em Kyushu J. Math.}, \textbf{56} (2002), 433--446.





\bibitem[Nak20a]{Nak20a}
N. Nakayama, Singularity of Normal Complex Analytic Surfaces Admitting Non-Isomorphic Finite Surjective Endomorphisms, Preprint, \href{https://www.kurims.kyoto-u.ac.jp/preprint/file/RIMS1920.pdf}{RIMS1920.pdf},  Kyoto Univ., July 2020.



\bibitem[Nak20b]{Nak20b}
N.Nakayama, On normal Moishezon surfaces admitting non-isomorphic surjective endomorphisms, Preprint, \href{https://www.kurims.kyoto-u.ac.jp/preprint/file/RIMS1923.pdf}{RIMS1923.pdf}, Kyoto Univ., September 2020.






\bibitem[NZ10]{NZ10}
N. Nakayama and D.-Q. Zhang,
  Polarized endomorphisms of complex normal varieties,
  {\em Math. Ann.}, \textbf{346} (2010), no. 4, 991--1018, \href{https://arxiv.org/abs/0908.1688v1}{ arXiv:0909.1688v1}.





\bibitem[PS89]{PS89}
K. H. Paranjape and V.~Srinivas,
Self-maps of homogeneous spaces,
{\em Invent. Math.}, \textbf{98} (1989), 425--444.




\bibitem[Pro94]{Pro94}
Y.G. Prokhorov, Compactifications of {${\bf C}^4$} of index {$3$}, {\em Algebraic geometry and its applications ({Y}aroslavl{\cprime}, 1992)}, 159--169, Aspects Math., E25, 1994.





\bibitem[Rei72]{Rei72}
M. A. Reid, Miles, The complete intersection of two or more quadrics, Doctoral dissertation, University of Cambridge, 1972.





\bibitem[Wi\'{s}94]{Wis94}
J. Wi\'{s}niewski, A report on Fano manifolds of middle index and $b_2\ge 2$, {\em Projective geometry with applications}, 19--26, Lecture notes in Pure and Appl. Math. \textbf{166}, Dekker, New York, 1994.


\bibitem[Yan21]{Yan21}
M. Yanis, Totally Invariant Divisors of non trivial endomorphisms of the Projective Space, {\em Geom. Dedicata}, \textbf{217} (2023), Article number: 79.






\bibitem[Zho21]{Zho21}
G. Zhong, Totally invariant divisors of int-amplified endomorphisms of normal projective varieties, {\em J. Geom. Anal.}, \textbf{31} (2021), no. 3, 2568--2593. 




\end{thebibliography}
\end{document}